\documentclass[11pt]{amsart}
\usepackage{fancyhdr}
\usepackage{txfonts}

\usepackage{amsmath,amsthm,amssymb,hyperref, mdframed}
\usepackage{mathrsfs} 
\usepackage{tikz}
\usepackage[all]{xy}
\usetikzlibrary{decorations.markings}
\usetikzlibrary{backgrounds,shapes}
\usepackage{subfig}

\usepackage{color}

\definecolor{aliceblue}{rgb}{0.9, 0.95, 1.0}

\newcommand{\C}{{\mathbb C}}
\newcommand{\D}{{\mathbb D}}
\newcommand{\HH}{{\mathbb H}}

\newcommand{\Te}{Teichm\"{u}ller }

\newcommand\T{{\mathcal T}}

\newtheorem{theorem}{Theorem}[section]
\newtheorem{prop}[theorem]{Proposition}

\newtheorem*{jthm}{Theorem}

\newtheorem{thm}{Theorem}[section]

\newtheorem{lem}[thm]{Lemma}
\theoremstyle{definition}
\newtheorem{defn}[thm]{Definition}

\newcommand{\cp}{\mathbb{C}\mathrm{P}^1}

\title[Quadratic differentials, measured foliations and metric graphs]{Quadratic differentials, measured foliations\\ and metric graphs on punctured surfaces}

\author{Kealey Dias}
\author{Subhojoy Gupta}
\author{Maria Trnkova}

\address{Department of Mathematics and Computer Science, Bronx Community College of the City University of New York, Bronx, NY 10453 USA.}

\email{kealey.dias@bcc.cuny.edu}

\address{Department of Mathematics, Indian Institute of Science, Bangalore 560012, India.}

\email{subhojoy@iisc.ac.in}

\address{Department of Mathematics, University of California, Davis, CA 95616, USA.}
\email{mtrnkova@math.ucdavis.edu}

\begin{document}

\begin{abstract}  
A meromorphic quadratic differential on a punctured Riemann surface induces horizontal and vertical measured foliations with pole-singularities. In a neighborhood of a pole such a foliation comprises foliated strips and half-planes, and its  leaf-space determines a metric graph. 
We introduce the notion of an asymptotic direction at each pole, and show that for a punctured surface equipped with a choice of such asymptotic data, any compatible pair of measured foliations uniquely determines a complex structure and a meromorphic quadratic differential realizing that pair.
This proves the analogue of a theorem of Gardiner-Masur, for meromorphic quadratic differentials. We also prove an analogue of the Hubbard-Masur theorem, namely, for a fixed punctured Riemann surface there exists a meromorphic quadratic differential with any prescribed horizontal foliation, and such a differential is unique provided we prescribe the singular-flat geometry at the poles.

\end{abstract}

\maketitle

\section{Introduction}
A holomorphic quadratic differential on a Riemann surface has associated coordinate charts with transition maps that are half-translations ($z\mapsto \pm z+c$). This induces a \textit{singular-flat structure} on the surface, namely, a flat metric with conical singularities, together with a pair (horizontal and vertical) of \textit{measured foliations}. These structures have been useful in  Teichm\"{u}ller theory, and the study of the mapping class group of a surface (see \cite{FLP}).
The correspondence between these analytical objects (the differentials) and their induced geometric structures is well-understood for a closed surface. In particular,  the work of Hubbard-Masur in \cite{HubbMas} proved that for a fixed compact Riemann surface  $X$ of genus $g\geq 2$,  assigning the induced horizontal (or vertical) foliation to a holomorphic quadratic differential defines a homeomorphism between the space $Q(X)$, and the space of measured foliations $\mathcal{MF}_g$. Moreover, in \cite{GardMas} Gardiner-Masur proved that the pair of horizontal and vertical foliations uniquely determines the complex structure and holomorphic quadratic differential inducing those foliations (see Theorem 3.1 in that paper).  
The main result in this article is the analogue of the Gardiner-Masur theorem for surfaces with punctures.

The analogue of the Hubbard-Masur theorem has been extended to the case of meromorphic quadratic differentials on a punctured surface $S$ of negative Euler characteristic, by the work in  \cite{GuptaWolf2} (that deals with higher order poles at the punctures) and \cite{GuptaWolf0}  (that deals with poles of order two). 
It is well-known that poles of order one (i.e. simple poles) can be reduced to the classical theory (i.e. the holomorphic case) by taking a branched double cover; we therefore consider poles of order greater than one throughout.  In the work of Gupta-Wolf, the behaviour of the measured foliations  at a pole-singularity was analyzed in terms of the ``principal part" of the differential.  This paper develops a more constructive approach in terms of various cut-and-paste operations; as an application we provide an alternative generalization of the Hubbard-Masur theorem in terms of the singular-flat geometry at the poles.

\subsection*{Sphere with $\leq 2$ punctures} We shall first focus on the case when $S$ is the sphere with at most two punctures, that is, $S= \C$ or $\C^\ast$. As a special case of the  classical Three-Pole theorem (see \cite{Jenkins2}),  the trajectory structure of a meromorphic quadratic differential on $S$ comprises foliated strips and half-planes. Thus, the induced measured foliations can be described in terms of their leaf-spaces, that are metric graphs on the punctured sphere. These metric graphs have $(n-2)$  infinite-length edges incident on any puncture of order $n\geq 3$, and a loop or infinite edge for each pole of order $2$ (see \S2.3 for details). 

The case that $S=\C$ was dealt with in \cite{AuWan2}; in this case the holomorphic quadratic differential necessarily has a pole of order $n\geq 4$ at $\infty$, and has the form $p(z) dz^2$ where $p$ is a polynomial of degree $n-4$.  By a conformal change of coordinates, it can be arranged that the polynomial is \textit{monic}, namely, that the leading coefficient is $1$, and \textit{centered}, namely, that the zeroes of the polynomial have vanishing mean. The leaf-spaces of the induced measured foliations are then planar trees, and the result for this case can be summarized as follows (see \S3 for details):

 \begin{jthm}[Au-Wan, \cite{AuWan2}]
 The space $\mathcal{MF}_0(n)$ of the measured foliations on $\cp$ with a single pole-singularity of order $n > 4$ at $\infty$ admits a bijective correspondence with the space $\mathcal{T}(n-2) $ of planar metric trees with $(n-2)$ labelled infinite rays incident at $\infty$, and $\mathcal{MF}_0(n) \cong \mathcal{T}(n-2) \cong \mathbb{R}^{n-5}$. 
 
Moreover, let $Q_0(n)\cong  \mathbb{C}^{n-5}$ be the space of monic and centered polynomial quadratic differentials of degree $n-4$. Then the map
\begin{equation*}\label{mapPhi}
\Phi_1: Q_0(n) \to \mathcal{MF}_0(n) \times \mathcal{MF}_0(n)
\end{equation*}
that assigns to a polynomial quadratic differential its associated horizontal and vertical foliations, is a homeomorphism.
\end{jthm}

\noindent \textit{Remark.} The space of measured foliations  $\mathcal{MF}_0(n)$ decomposes into regions corresponding to the different combinatorial types of planar trees with labelled ends, and there is exactly a Catalan number of them. This is closely related to the classification of the trajectory-structure for polynomial vector fields on $\mathbb{C}$ (see \cite{BranDias}, \cite{Dias}, \cite{DES} and the references therein). One of the differences is that a foliation induced by a quadratic differential is typically not orientable.   \\

For $S= \C^\ast$, let $n,m\geq 2$ denote the orders of poles at $0$ and $\infty$ respectively.  A meromorphic quadratic differential on $\C^\ast$ with poles of these prescribed orders has the form 
\begin{equation}
q= \frac{p(z)}{z^n} dz^2
\end{equation}
where $p$ is a polynomial of degree $n+m-4$ such that $p(0) \neq 0$.

As we shall see in \S2.1, the argument of the leading order coefficient at the poles determines the asymptotic directions of the induced foliations at the poles. At a pole of order two, this asymptotic direction is the ``slope" of the leaves when lifted to the universal cover $\HH$ of a neighborhood of the pole. At each higher order pole, the asymptotic direction of a single horizontal leaf determines the complete set of asymptotic directions of horizontal as well as vertical leaves.  If we prescribe this asymptotic data,  the remaining coefficients of $p(z)$ and the  modulus of the leading order term at each pole, parametrize the space of such quadratic differentials $Q_0(n,m) \cong \mathbb{R} \times \mathbb{C}^{n+m-5}\times \mathbb{R} $ provided $n+m >4$. 

Our first result is:

\begin{thm}\label{thm1}  Let $n,m\geq 2$ such that $n,m$ are not both equal to $2$. Let $\mathcal{MF}_0(n,m)$ be the space of  measured foliations on $\cp$ with a pole-singularity of order $n$ at $0$ and of order $m$ at $\infty$,  with prescribed asymptotic data at the poles. 
Let 
\begin{equation*}
\Phi_2: Q_0(n,m) \to \mathcal{MF}_0(n,m) \times \mathcal{MF}_0(n,m)
\end{equation*}
be the map that assigns to a quadratic differential with prescribed asymptotic data,  its induced horizontal and vertical measured foliations. Then $\Phi_2$ defines a homeomorphism to the  subspace comprising pairs of foliations that 
\begin{itemize} 
\item do not both have transverse measure zero around the punctures, and 
\item in case that either $n$ or $m$ equals two, and both foliations have positive transverse measures around the punctures, then the two transverse measures are compatible with the prescribed asymptotic direction at the pole of order two (see Definition \ref{compat}). 
\end{itemize}
\end{thm}

\medskip

The key part in the proof of Theorem \ref{thm1} is defining an inverse map to $\Phi_2$ (see \S4.2). This  uses a decomposition of the measured foliations on $\C^\ast$ into ``model foliations" on neighborhoods of the two punctures, and the remaining annulus.  The desired meromorphic quadratic differential is then constructed by assembling the singular-flat surfaces that realize the corresponding pairs of foliations on each of these subsurfaces. On a punctured disk, realizing such a pair of model foliations crucially uses the work of Au-Wan from \cite{AuWan2}.  The special case when $n=m=2$ is discussed in \S4.4.

\subsection*{Surface of negative Euler characteristic} Now consider the case when $S$ is an oriented surface of genus $g$ and $k$ labelled punctures, such that the Euler characteristic $2-2g-k<0$. Let $\mathfrak{n} = (n_1, n_2, \ldots, n_k)$ be a $k$-tuple of integers, each greater than one. Let $\mathcal{MF}_g(\mathfrak{n})$ be the space of measured foliations with a pole-singularity of order $n_i$ at the $i$-th puncture, and with prescribed asymptotic directions at the poles.
Combining the results in  \cite{GuptaWolf0} and  \cite{GuptaWolf2},  we  parametrize this space in \S2.2 (see Proposition \ref{mfgn-prop}). This work of Gupta-Wolf had also defined these spaces, but had done so relative to fixing a choice of a ``disk neighborhood" of the poles;  the notion of asymptotic data of the foliations at the poles, introduced in this paper, provides a cleaner definition.

By the work in \cite{Bridgeland-Smith}, a \textit{generic} measured foliation in $\mathcal{MF}_g(\mathfrak{n})$ comprises foliated strips and half-planes, and thus has a leaf-space that can be represented as an embedded metric graph on the surface, exactly as in the Three-Pole case.   There are, however, measured foliations with more complicated trajectory structure (e.g. dense leaves) whose corresponding leaf-space is described as a $\pi_1(S)$-invariant $\mathbb{R}$-tree in the universal cover of $S$, with an additional  $\pi_1(S)$-invariant collection of infinite rays corresponding to the higher-order poles (see \S3.3. of \cite{GuptaWolf2}). 

Let  ${Q}_g(\mathfrak{n})$ be the space of meromorphic quadratic differentials on $S$, with a pole of order $n_i$ at the $i$-th puncture.  Our main result is: 

\begin{thm}\label{thm2} 
Let  $S$ be an oriented surface of genus $g$ and $k$ punctures such that $2-2g-k<0$. Let $\mathfrak{n} = (n_1, n_2, \ldots, n_k)$ be a $k$-tuple of positive integers, each greater than one, and fix a set $\mathfrak{a}$ of asymptotic data comprising a tangent direction at each pole. 

 Let  $(\mathcal{H}, \mathcal{V}) \in  \mathcal{MF}_g(\mathfrak{n}) \times \mathcal{MF}_g(\mathfrak{n})$  be a compatible pair of transverse measured foliations, that is, 
\begin{itemize}
\item $\mathcal{H}$ has prescribed asymptotic directions given by $\mathfrak{a}$ at the poles, and $\mathcal{V}$ has the opposite set $\sqrt{-1}\cdot  \mathfrak{a}$ of asymptotic directions (see Definition \ref{opp}), 
\item $\mathcal{H}$ and $\mathcal{V}$ do not simultaneously have transverse measure zero around any puncture, and 
\item if $n_i=2$, and both $\mathcal{H}$ and $\mathcal{V}$ have positive transverse measure around the $i$-th puncture, then the two transverse measures are compatible with the prescribed asymptotic direction at the order two pole (see Definition \ref{compat}). 
\end{itemize} 

Then there exists a unique meromorphic quadratic differential in $Q_g(\mathfrak{n})$ that induces a horizontal foliation equivalent to  $\mathcal{H}$ and vertical foliation equivalent to  $\mathcal{V}$.
\end{thm}

The proof of Theorem \ref{thm2} in \S5.1 uses a decomposition of the desired pair of measured foliations to model foliations around each puncture (as in the proof of Theorem \ref{thm1}) together with a pair  of measured foliations on a surface with boundary. Realizing the latter pair can be reduced to the case of a closed surface by doubling across the boundary; however, the final  assembly of singular-flat surfaces requires the angle at which either foliation intersects the boundary to be prescribed. This is achieved in an intermediate step that involves truncating cylindrical ends that are attached to each boundary component. \\

The space of meromorphic quadratic differentials ${Q}_g(\mathfrak{n})$  forms a vector bundle over the ``appended Teichm\"{u}ller space"  $\widehat{\T}_{g,k}$ of conformal structures on $S$ up to isotopy fixing a framing of the tangent-space at the punctures (see Definition 3.3. of  \cite{GuptaMj1}). Here, the space  $\widehat{\T}_{g,k}$  records,  in addition to the $6g-6+2k$ parameters of the \Te space of $S$, a real ``twist" parameter at each puncture.  Let $\pi: {Q}_g(\mathfrak{n})\to \widehat{\T}_{g,k}$ be the projection map; any fiber $\pi^{-1}(X)$ comprises quadratic differentials  that are meromorphic with respect to the Riemann surface structure on $X$ and induces foliations that have asymptotic directions and integer twist parameters around each pole determined by the corresponding twist parameter on $X$ (\textit{c.f.} \S3.1 of \cite{GuptaMj1}, and see \S5.2 for details.)  

Our final result is a generalization of the Hubbard-Masur theorem to the case of meromorphic quadratic differentials on punctured surfaces. This generalization was first proved in \cite{GuptaWolf0} (for order-two poles) and in \cite{GuptaWolf2} (for higher-order poles) using the theory of harmonic maps, and their work uses the complex-analytic notion of a ``principal part" of a quadratic differential at each pole, with respect to a choice of a coordinate disk. The following alternative generalization  instead uses the space of ``model foliations"  $\mathcal{P}_n$ in the neighborhood of a pole of order $n$ (introduced in \S2.3). 

\begin{thm}\label{thm3} Let $S, \mathfrak{n}$ be as in Theorem \ref{thm2}.
Let ${X} \in \widehat{\T}_{g,k}$, and fix a measured foliation $\mathcal{H} \in  \mathcal{MF}_g(\mathfrak{n})$, and model foliations $F_i \in \mathcal{P}_{n_i}$ for each $1\leq i\leq k$. 
Suppose the asymptotic directions  $\mathfrak{a}$ and real twist parameters of $\mathcal{H}$ at the poles are those determined by the twist parameters of $X$, and $\mathcal{H}$ restricts to the model foliations $F_i^H \in \mathcal{P}_{n_i}$ in a disk $D_i \cong \mathbb{D}^\ast$ around the $i$-th pole,  where each pair $(F_i^H, F_i)$ is compatible, exactly as in Theorem \ref{thm2}. 
 Then there is a unique meromorphic quadratic differential $q\in Q_g(\mathfrak{n})$ satisfying $\pi(q)=X$, such that  the horizontal foliation of $q$ is equivalent to $\mathcal{H}$, and the vertical foliation of $q$ restricts to the model foliation $F_i$ at the $i$-th pole.
\end{thm}

\smallskip

A  key step in the proofs of Theorems \ref{thm1} and \ref{thm2} was the fact that a pair of model foliations in $\mathcal{P}_n$ uniquely determines a singular-flat metric on a neighborhood of that pole (see Proposition \ref{prop2}). Thus in Theorem \ref{thm3}, since $\mathcal{H}$ determines the horizontal model foliation at each pole,  prescribing the vertical model foliations is equivalent to prescribing the geometry of the singular-flat end corresponding to each pole. The strategy of the proof of Theorem \ref{thm3} in \S5.2 is to reduce to the case when all poles have order two (and all ends are cylindrical) and use the main result of \cite{GuptaWolf0}.

\bigskip

\textbf{Acknowledgements.} This article has been in the works for several years, SG and MT are grateful for the support by NSF grants DMS-1107452, 1107263, 1107367 ``RNMS: GEometric structures And Representation varieties" (the GEAR Network). SG is also grateful for the support by the Danish National Research Foundation centre of Excellence, Centre for Quantum Geometry of Moduli Spaces (QGM), the Department of Science and Technology (DST) MATRICS Grant no. MT/2017/000706, the Infosys Foundation, and the UGC. The authors are grateful to Fred Gardiner, as well as an anonymous referee, whose comments improved a previous version of the paper.

\section{Preliminaries} 

\subsection{Quadratic differentials and their induced geometry}

A holomorphic quadratic differential $q$ on a Riemann surface $X$ is a holomorphic section of the  symmetric square of canonical bundle $K_X^2$. Locally, such a holomorphic quadratic differential can be expressed as $q(z)dz^2$ where $q(z)$ is a holomorphic function. A holomorphic quadratic differential induces a singular-flat metric and horizontal and vertical foliations on the underlying Riemann surface that we now describe. For an account of what follows, see \cite{Streb} or \cite{Gard}. A key new notion introduced in this paper is that of an ``asymptotic direction" at a pole -- see Definition \ref{adata}.

  \begin{defn}[Singular-flat metric] A holomorphic quadratic differential induces a conformal metric locally of the form $\lvert q(z)\rvert \lvert dz\rvert^2$, which is a flat Euclidean metric with cone-type singularities at the zeroes, where a zero of order $n$ has a cone-angle of $(n+2)\pi$.
\end{defn}

\begin{defn}[Horizontal and vertical foliations]\label{fols}  A holomorphic quadratic differential on $X=\mathbb{C} \text{ or } \mathbb{C}^\ast$ determines  a  bilinear form  $q: T_xX \otimes T_x X \to \mathbb{C}$ at any point $x\in X$ away from the poles.  Away from the zeroes, there is a unique (un-oriented) \textit{horizontal direction} $v$ where $q(v,v)\in \mathbb{R}^{+}$. Integral curves of this line field on $X$ determine the \textit{horizontal foliation} on $X$. Similarly, away from the zeroes, there is a unique (un-oriented) \textit{vertical direction} $h$ where $q(h,h)\in i\mathbb{R}^{+}$. Integral curves of this line field on $X$ determine the \textit{vertical foliation} on $\mathbb{C}$.
\end{defn}

\noindent \textit{Remarks.} 1. The terminology arises from the fact that for the quadratic differential $dz^2$ on any subset of $\mathbb{C}$ (equipped with the coordinate $z$), the horizontal and vertical foliations are exactly the foliations by horizontal and vertical lines.\\
2. Conversely, if we start with (possibly non-compact) domains on $\C$ whose boundaries comprise straight line intervals  and  identify pairs of such geodesic by half-translations to obtain an oriented surface, then the resulting surface acquires a Riemann surface structure as well as a holomorphic quadratic differential. The latter descends from the standard differential $dz^2$ on the domains, since $dz^2$ is invariant under half-translations. The condition that the boundary edges are identified by half-translations is equivalent to the requirement that the identification is by a (Euclidean) isometry and the horizontal foliation of the standard differential $dz^2$ intersects any pair of boundary edges being identified at the same angle.

\begin{defn}[Prong-singularities and natural coordinates]\label{pp} At the zero of order $k\geq 1$ of a quadratic differential $q$, the horizontal (and vertical) foliation has a  \textit{$(k+2)$-prong singularity}. That is, in a neighborhood of the zero, the horizontal  foliation is the pullback of the horizontal foliation on $\mathbb{C}$ by the map $z\mapsto \xi = z^{k/2 +1}$ (which is a branched cover, branched at the zero of the target $\xi$-plane). Here, $\xi$ is called the \textit{natural coordinate} for the quadratic differential, since $q= d\xi^2$ (up to a constant multiplicative factor).
\end{defn}

  \begin{figure}
    \centering
 \includegraphics[scale=0.37]{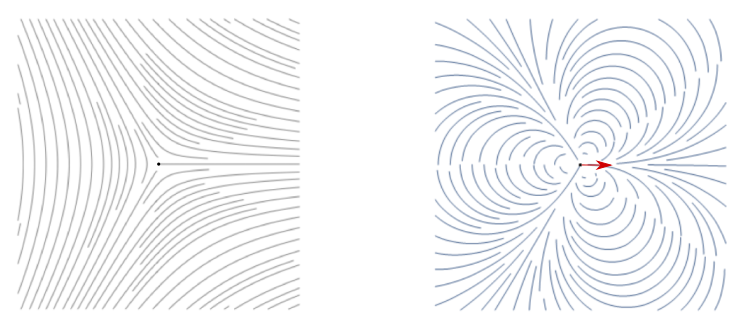}\\
  \caption{The horizontal foliation for $zdz^2$ has a $3$-prong singularity at the origin (left), and a pole-singularity of order $5$ at infinity (right). The red arrow shows a choice of an asymptotic direction at the pole. }
  \label{singfig}
\end{figure}

\begin{defn}[Pole-singularities of higher order]\label{singh} 
At a pole of order $n\geq 2$, the foliation induced by $q$ has a \textit{pole-singularity of order $n$}.  For $n>2$, the  induced singular-flat geometry comprises $(n-2)$ foliated Euclidean half-planes surrounding the pole in cyclic order; the horizontal leaves are asymptotic to $(n-2)$ directions at the pole, and the same for vertical leaves. See Figure 1, and \S6 of \cite{Streb} for details. Indeed, if the leading order term for $q$ is $ \frac{a^2}{z^n}$ in some local coordinate $z$ around the pole, for some $a \in \C^\ast$ with $\textit{Arg}(a) = \theta$, then the horizontal leaves are asymptotic to the directions at angles $\theta + j \cdot \frac{2\pi}{n-2}$ where $0\leq j <n-2$ and the vertical leaves are asymptotic to the directions  $\theta + (j + \frac{1}{2}) \cdot \frac{2\pi}{n-2}$. \end{defn} 

\begin{defn}[Pole-singularity of order $2$]\label{sing2} Around a pole of order two, the induced foliation looks either like a foliation by concentric circles, or leaves spiralling to the pole. That is, one can choose a local coordinate disk $U \cong \mathbb{D}^\ast$ around the pole such that $q = -\frac{a^2}{z^2} dz^2$ for some $\pm a\in \C^\ast$, called the \textit{residue} at the pole, which is in fact coordinate-indpendent. The case of concentric circles then arises for the horizontal foliation when $a^2 \in \mathbb{R}^+$, and for the vertical foliation when $a^2 \in \mathbb{R}^-$.   In either case, in the singular-flat metric induced by $q$, a neighborhood of the pole is isometric to a semi-infinite Euclidean cylinder. (See Chapter III \S7.2 of \cite{Streb}, and \S2.2 of \cite{GuptaWolf0}.)  \\
We also note that in the universal cover $p:\mathbb{H} \to \mathbb{D}^\ast$ given by $w\mapsto z=e^{2\pi i w}$, $q$ pulls back to the quadratic differential $\frac{a^2}{4\pi^2} dz^2$, and the induced foliation on $\mathbb{H}$ is by straight lines at an angle  $\theta = - \text{Arg}(a)$. (See Figure 2.)  This will be the definition of the asymptotic direction in this case (see Definition \ref{adata}). 
\end{defn}

 \begin{figure}
    \centering
 \includegraphics[scale=0.45]{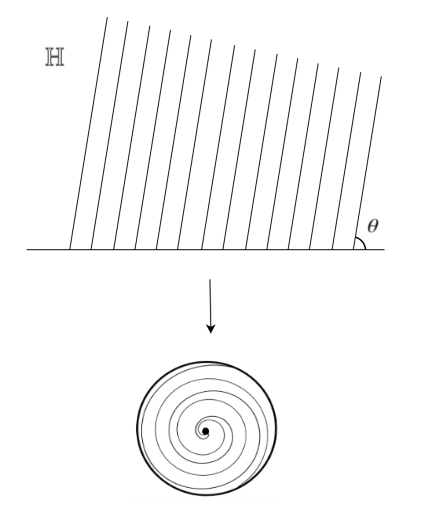}\\
  \caption{The angle $\theta$ of the leaves in $\mathbb{H}$ determines the asymptotic direction of the foliation at a pole-singularity of order two in $\D^\ast \cong \HH/\mathbb{Z}$. }
  \label{singfig}
\end{figure}

\noindent \textit{Remark.}  In a neighborhood of a pole of order $1$, also called a \textit{simple} pole,  this foliation looks like a ``fold", since it is the pullback of the horizontal foliation by the map $z\mapsto \xi = \sqrt z$. As alluded to in the Introduction, this implies that the pole-singularity becomes a regular point on the double cover branched at the simple pole.

\begin{defn}[Transverse measure]\label{tm} The horizontal (resp. vertical) foliation induced by a holomorphic quadratic differential is equipped with a transverse measure, that is, any  arc transverse to the foliation acquires a measure that is invariant under transverse homotopy of the arc. Namely, the  \textit{transverse measure} of such an arc $\gamma$  transverse to the horizontal foliation is
$$\tau_h(\gamma) =  \lvert \displaystyle\int\limits_\gamma\Im (\sqrt q) (z)  dz  \lvert $$
assuming $\gamma$ is contained in a coordinate chart, and similarly the transverse measure $\tau_v(\gamma)$ of an arc $\gamma$ transverse to the \textit{vertical} foliation  is given by the modulus of the integral of the real part $ \Re (\sqrt q) (z) $.
In general one adds such distances along a cover of the arc comprising of coordinate charts; this is well-defined as the above integrals are preserved (up to sign) under change of coordinates.
Given a simple closed curve $\gamma$ that is homotopically non-trivial, we define the transverse measure of the homotopy class $[\gamma]$ to be the infimum of the transverse measures of curves homotopic to $\gamma$.
\end{defn}

These foliations equipped with a transverse measure induced by a holomorphic quadratic differential are examples of a \textit{measured foliation} on a smooth surface, that is defined purely as a topological object as follows:

\begin{defn}[Measured foliations]\label{mfdef}  A \textit{measured foliation} on a (possibly punctured) smooth surface $S$ is a $1$-dimensional foliation that is smooth except finitely many prong-singularities (see Definition \ref{pp}), equipped with a transverse measure.  We shall define two such measured foliations to be \textit{equivalent} if they differ by an isotopy and Whitehead-moves.  If the surface has punctures, then the isotopy is relative to the punctures, in the sense that  a choice of framing  of the tangent space at the pole (given by a tangent direction $v$  and an orthogonal  vector $\sqrt{-1}\cdot  v$) is kept fixed by the isotopy. Equivalently, if we  consider a real oriented blow-up of each puncture to a boundary circle to obtain a surface-with-boundary, then the isotopy is required to fix each boundary component pointwise. 
\end{defn}

\noindent \textit{Remarks.} 1.  For a \textit{closed} surface, it can be shown that two measured foliations are equivalent if and only if the respective transverse measures of homotopy classes of all simple closed curves are equal. See \cite{FLP} for a comprehensive account of this. \\
2. As mentioned in the Introduction, for a closed Riemann surface $X$, \cite{HubbMas} showed that any such equivalence class of a measured foliation (on the underlying smooth surface) is in fact the horizontal (or vertical) foliation of a unique holomorphic quadratic differential. \\

The following fact about the global trajectory-structure is well-known (see \cite{Jenk} or \cite{Streb}):

\begin{prop}\label{fstruc} Let $F$ be a  measured foliation on a compact surface $S$ with finitely many pole-singularities of order greater than one at the punctures. Then the surface can be decomposed into finitely many regions, such that the restriction of $F$ to any region yields one of the following:
 \begin{enumerate}
 \item a foliated half-plane, foliated by leaves parallel to the boundary.
 \item a foliated strip, foliated by leaves parallel to the two boundary components. or
  \item  a foliated annulus, with leaves that are closed curves parallel to the two boundary components. (We shall continue to call this a ``ring-domain".)
 \item  a \textit{spiral domain} in which each leaf is dense.
 \end{enumerate}
\end{prop}
\medskip

At the pole-singularities of a measured foliation, we introduce  a circle-valued parameter:

\begin{defn}[Asympotic direction]\label{adata}  The \textit{asymptotic direction} of  measured foliation $F$ at a pole-singularity of order $n>2$ is the asymptotic direction $\theta$ of a leaf at the pole, where $\theta$ can be thought of as a point on the unit tangent circle at the pole. 
At a pole-singularity of order $n=2$, the asymptotic direction of $F$ is defined to be the angle $\theta \in [0,\pi)$ of the linear foliation on the  universal cover $\mathbb{H}$  that descends to a foliation equivalent to $F$  in a punctured-disk neighborhood of the pole (\textit{c.f.} Definition \ref{sing2}). Note that in the case that $\theta = 0$ the foliation $F$ comprises closed leaves (concentric circles) around the pole; in this case the transverse measure around the pole is necessarily zero. 
\end{defn}

\subsection{Compatible pairs} 
The horizontal and vertical foliations induced by a meromorphic quadratic differential are, by construction, transverse to each other away from the prong and pole singularities. In this section we list some such ``compatibility" criteria that are necessary for a pair of measured foliations to be equivalent to the horizontal and vertical foliations of some meromorphic quadratic differential. 

First, the ``transversality"  of the two foliations implies the following:

\begin{lem}
	\label{transverse0}
	Let $\mathcal{H}$ and $\mathcal{V}$ be the horizontal and vertical foliations, respectively, of a meromorphic quadratic differential on some surface $S$. For any simple closed curve $\gamma$ on $S$ that is homotopically non-trivial, let  $\tau_h,\tau_v$ be the transverse measures of the homotopy class of $\gamma$, for $\mathcal{H}$ and $\mathcal{V}$ respectively. Then $\tau_h, \tau_v$ cannot both be zero.
\end{lem}
\begin{proof}

It suffices to show that if the vertical foliation has a ring domain with core curve $\gamma$, then the horizontal foliation cannot have a ring domain with the same core curve.

Suppose both the horizontal and vertical foliations have ring domains, with core curves homotopic to $\gamma$.
Let $\gamma_v$ be a leaf in the vertical ring domain, and $\gamma_h$ be a leaf in the horizontal ring domain.
There are two cases:

Case 1:  The leaves $\gamma_v$ and $\gamma_h$ are disjoint:

Since they are homotopic to each other, they bound an annulus $A$ between them. Consider the restriction $F$ of, say, the horizontal foliation on $A$.  One boundary component of $A$ is a leaf of $F$, and the other boundary (which is a vertical leaf) is transverse to $F$.  This implies that $F$ must have singularities in $A$; however any prong-singularity has negative index, and since the Euler characteristic of $A$ is zero,  we again have a contradiction to the Poincar\'{e}-Hopf theorem. 

Case 2:  The leaves $\gamma_v$ and $\gamma_h$ intersect:

By an ``innermost disk" argument we can choose two sub-arcs of $\gamma_v$ and $\gamma_h$ respectively,  that bound a topological disk $D$.  The horizontal foliation is transverse to the part of the boundary $\partial D$ that is vertical; we can assume, after an isotopy, that the leaves intersect the boundary orthogonally. Then doubling across it, we obtain a foliated disk such that the boundary is a leaf, and the foliation has only prong-type singularities. This  contradicts the Poincar\'{e}-Hopf theorem, exactly as in the proof of the Claim above. 

This contradicts our assumption that both the horizontal and vertical foliations have ring domains with core curve $\gamma$; hence, one the transverse measures $\tau_h, \tau_v$ of the homotopy class of $\gamma$ is positive. \end{proof}

We shall apply the above lemma, in particular, for asserting  the transverse measures of the horizontal and vertical foliations around any pole-singularity (considered as a puncture on the surface) cannot both be zero. 
At these pole-singularities, the horizontal and vertical foliations satisfy some additional compatibility conditions, which we now define:

\begin{defn}[Opposite parameters]\label{opp} 
Given an asymptotic direction $a \in S^1$ at a pole of order $n>2$ (See Definition \ref{adata}), a direction $a^\prime \in S^1$ is \textit{opposite} if it differs from $a$ by an odd multiple of $\pi/(n-2)$. 
 For an asymptotic direction $\theta$ at a pole of order two, the opposite is the asymptotic direction $\theta + \pi/2$ (modulo $\pi$). 

Note that the horizontal and vertical foliations induced by a meromorphic quadratic differential  $q$ have opposite asymptotic directions at each pole (see Definition \ref{singh}). 

\end{defn} 

\begin{defn}[Compatible transverse measures]\label{compat} Let $F,G$ be two measured foliations with a pole-singularity of order two and asymptotic direction $\theta \in (0,\pi)$, such that the transverse measures $\tau_F$ and $\tau_G$ around the pole are positive. Then these transverse measures are said to be \textit{compatible with the asymptotic direction} $\theta$, if  the ratio $\tau_F/\tau_G = \lvert \tan{\theta} \rvert$.
\end{defn}

\noindent \textit{Remark.} As in the previous definition, the motivation for this definition is that this compatibility of transverse measures  is necessary if $F,G$ are the horizontal and vertical foliations induced by $q = -\frac{a^2}{z^2} dz^2$. Indeed, from our definitions, in that case $\theta = -\text{Arg}(a)$ (modulo $\pi$), and $\tau_F = \lvert a \rvert \lvert \cos \theta \rvert$ and $\tau_G = \lvert a \rvert   \sin \theta $. \\

Finally, we shall say:

\begin{defn}\label{compat2}  Two measured foliations $\mathcal{H}$ and $\mathcal{V}$ on a punctured surface $S$ with pole-singularities at the punctures of identical orders, are said to be \textit{compatible} if 
\begin{itemize}
\item[(i)] they are transverse to each other away from the prong and pole-singularities,
\item[(ii)] the transverse measures around any pole-singularity are not both zero, 
\item[(iii)] the asymptotic directions at each pole-singularity are opposite, and 
\item[(iv)] if both have positive transverse measures around a pole-singularity of order two, then the two transverse measures are compatible with the asymptotic direction, as in Definition \ref{compat}. 
\end{itemize} 
\end{defn}

\subsection{Space of measured foliations} 
We can  define the following space of measured foliations (already introduced in the Introduction):

\begin{defn}\label{mfgn-def} For an integer $k\geq 1$, and an integer $k$-tuple $\mathfrak{n}= (n_1,n_2,\ldots, n_k)$ such that each $n_i\geq 2$, we define  $\mathcal{MF}_g(\mathfrak{n})$ to be the space of (equivalence classes of) measured foliations on an oriented surface $S$ of genus $g$ and $k$ labelled points, such that the $i$-th point is a pole-singularity of order $n_i$, and the asymptotic direction (see Definition \ref{adata}) at each point is prescribed. 
\end{defn}

In this section shall describe the parametrization of $\mathcal{MF}_g(\mathfrak{n})$, following the discussions in \cite{GuptaWolf0} and   \cite{GuptaWolf2}.  The topology on  $\mathcal{MF}_g(\mathfrak{n})$ will also be described in the proof of Proposition \ref{mfgn-prop}. We shall start with:

\begin{defn}[Model foliations on $\mathbb{D}^\ast$]\label{model}  
For any pole-singularity of order $n>2$, there is an (open) punctured disk neighborhood $U \cong \mathbb{D}^\ast $ that is (a) a `sink-neighborhood", that is, any leaf entering $U$ continues to the pole, after possibly passing through prong-singularites, and (b) satisfies the property that no leaf exits $U$ and then enters $U$ again.  (See the discussion in Definition 12 of \cite{GuptaWolf2}.)  The measured foliation $F\vert_U$ is then a ``model foliation" for that order $n$ of a pole-singularity. Note that  $F\vert_U$ comprises the foliated half-planes and (possibly) foliated strips (either infinite, from the puncture to itself, or semi-infinite, from the puncture to $\partial U$). As before, we consider such foliations up to an equivalence: two model foliations on $U$ are equivalent if they differ by Whitehead moves, or an isotopy that fixes a framing of the tangent-space at the puncture (but is allowed to move points on the boundary $\partial U$).\end{defn} 

Recall that the leaf-space $G$  of a measured foliation $F$ on a surface is defined as $$G := X/\sim$$ where $x\sim y$ if $x,y$ lie on the same leaf of $F$, or on leaves that are incident on a common prong-singularity.  The leaf-space of $F\vert_U$ is then a metric graph $G$ with finitely many vertices and edges, where the finite-length edges of $G$ are the leaf-spaces of the strips, and the $(n-2)$ infinite-length rays  are the leaf-spaces of the half-planes. See Figure 3. Conversely, given such a metric graph $G$, it is easy to construct a model foliation on a punctured disk with leaf-space $G$; this is uniquely defined once the asymptotic direction at the puncture is fixed.

 \begin{figure}
    \centering
 \includegraphics[scale=0.33]{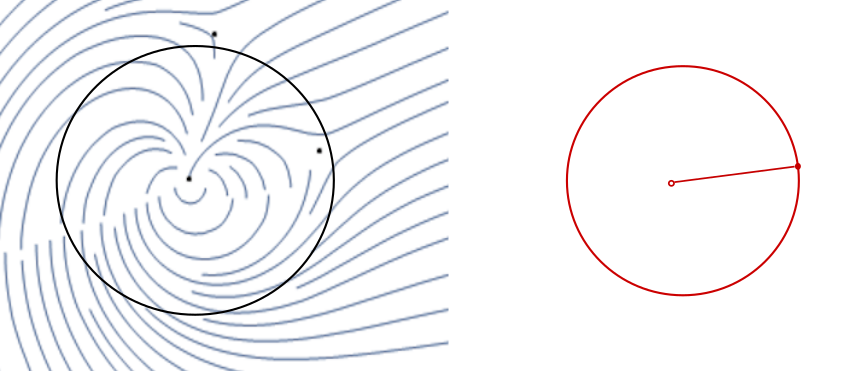}\\
  \caption{A model foliation in a neighborhood of a pole of order 3 (inside the circle shown on the left) has a leaf-space that is a metric graph (show on the right) with one finite length cycle and one infinite ray. }
  \label{singfig}
\end{figure}

For $n>2$, let ${\mathcal{P}}_n$ be the space of  model foliations on a punctured disk with a pole-singularity of order $n$ at the puncture, with a prescribed asymptotic direction at the puncture.  We can equip this space with the topology on the space of their leaf-spaces: two metric graphs are close if they combinatorially equivalent up to Whitehead moves on short edges, and the lengths of the finite-length edges are close. Following Proposition 17 of \cite{GuptaWolf2} we have:

\begin{prop}\label{prop-pn}  Let $n>2$. The space ${\mathcal{P}}_n$ is homeomorphic to $\mathbb{R}^{n-3} \times \mathbb{R}_{\geq  0}$  where the first factor is parametrized by the edge-lengths of the leaf-spaces, the second factor records the  transverse measure of $\partial U$.
\end{prop}

\subsubsection*{Order two pole} For a pole of order $n=2$, there is a punctured disk neighborhood  $U$ of the pole such that the leaf-space of $F\vert_U$ is either an infinite ray (in the case that the transverse measure is zero), or a circle of circumference equal to the transverse measure around $\partial U$. The foliation on $U$ is rotationally symmetric; however we can define the asymptotic direction of the leaves on the universal cover (see Definition \ref{sing2}). Once again, there is a space $\mathcal{P}_2$ of such model foliations on a punctured disk, where we note:

\begin{prop}\label{prop-p2}  A model foliation on $\mathbb{D}^\ast$ with a pole singularity of order $2$ at the puncture is uniquely determined by the transverse measure, and the asymptotic direction at the pole.  \end{prop} 
\begin{proof}
Recall from Definition \ref{adata} that the asymptotic direction at the pole determines the angle $\theta \in [0,\pi)$  of the leaves of the straight-line foliation on the universal cover $\mathbb{H}$. Thus the lift of the model foliation to the universal cover is specified completely by the asymptotic direction.  If the asymptotic direction is $0$, then the transverse measure is necessarily equal to zero, and the lifted foliation is by horizontal lines, which is the horizontal foliation of the quadratic differential $\tilde{q} = dz^2$ on $\HH$.  
Otherwise, the lifted foliation is the horizontal foliation of a constant quadratic differential $\tilde{q} = a^2dz^2$ on $\HH$, where $a\in \C^\ast$ satisfies $\text{Arg}(a) = -\theta$. If the transverse measure is prescribed to be $\tau>0$, then from our definitions $\tau = \lvert a \rvert \lvert \cos \theta \rvert$, and hence $a$ is uniquely determined. 
Thus, given $\theta$ and the transverse measure, the measured foliation in the quotient $\mathbb{D}^\ast = \HH/\langle z \mapsto z+1\rangle$ is uniquely determined. 
\end{proof}

\subsubsection*{Foliations on a surface with boundary}

In  \S3.1 of \cite{GuptaWolf2}  and  \S3.4 of \cite{ALPS}, the space of measured foliations  $\mathcal{MF}_{g,k}$ on a compact oriented surface of genus $g$ and $k\geq 1$ boundary components, and negative Euler characteristic, was parametrized.  In their work the foliations were considered up to isotopy that allowed points on the boundary to move, that is, there was no ``twist" parameter associated with the boundary components.   They proved (see, Proposition 3.9 of \cite{ALPS} or Proposition 11 of \cite{GuptaWolf2}) that:

\begin{prop}\label{mfb} The space of  measured foliations  $\mathcal{MF}_{g,k}$ is homeomorphic to $\mathbb{R}^{6g-6 + 3k}$.
\end{prop}

 Here, the topology on  $\mathcal{MF}_{g,k}$ is such that two measured foliations are close if the transverse measures of (homotopy classes) of a  filling set of arcs or simple closed curves on $S$ are close. 
Note that in their parametrization, the transverse measure around a boundary component determines a \textit{real}-valued parameter $\tau$ (and not a non-negative real parameter); $\tau<0$ is interpreted as the foliation having a ring domain adjacent to the boundary (i.e.\ a cylinder foliated by closed leaves parallel to boundary) with transverse measure $\lvert \tau\rvert$. 

\subsubsection*{Parametrizing $\mathcal{MF}_g(\mathfrak{n})$}

The measured foliations with pole-singularities in  $\mathcal{MF}_g(\mathfrak{n})$ have an additional real-valued twist parameter associated with each pole, since we consider foliations up to an isotopy that fixes a framing of the tangent space at each such point. Such a framing is determined by the asymptotic data at the pole and its opposite (see Definitions \ref{adata} and \ref{opp}).

Combining the cases of foliations on a surface-with-boundary, and on a punctured disk, as discussed above, we have:

\begin{prop}\label{mfgn-prop} Let $S$ be an oriented surface of negative Euler-characteristic, having genus $g$ and $k\geq 1$ punctures. Let $\mathfrak{n}$ be a $k$-tuple of integers greater than one, as in Definition \ref{mfgn-def}. Then the space of measured foliations  $\mathcal{MF}_g(\mathfrak{n})$  is homeomorphic to $\mathbb{R}^\chi$ where $\chi = 6g-6 + \sum\limits_{i=1}^k (n_i +1)$.
\end{prop}

\begin{proof}
Let $F$ be a measured foliation in $\mathcal{MF}_g(\mathfrak{n})$. 
Deleting the neighborhoods $U_1,U_2,\ldots, U_k$ where $F$ restricts to a model foliation, we obtain a measured foliation $F_0$ on the surface-with-boundary $S^\prime = S \setminus (U_1 \cup U_2 \cup \cdots \cup U_k)$ that by Proposition \ref{mfb} is parametrized by $6g-6+3k$ parameters. 

On a punctured disk $U_i$ around the $i$-th puncture with a pole-singularity of order $n_i>2$, the model foliation $F\vert_{U_i}$ is specified by $n_i-3$ additional parameters by Proposition \ref{prop-pn} (since the transverse measure parameter of $\partial U_i$ has to coincide with that of $F_0$).  We have an additional  real twist parameter around each puncture, which measures the gluing of  $U_i$ with the corresponding boundary component of  $S^\prime$.  This is relevant only when the transverse measure around $\partial U_i$ is positive, since otherwise $\partial U_i$ is a closed leaf of the foliation, and foliations differing by a twist around it are in fact isotopic. 

The data of the twist parameter $\sigma_i$ can be thought of as measuring an additional twist associated with the boundary component $\partial U_i$ on $S^\prime$. As usual for Fenchel-Nielsen parameters, these twist parameters can be measured relative to a collection of reference arcs, each non-trivial in homotopy, between the boundary components of $S^\prime$.  
 Moreover, each twist parameter can be combined with the transverse measure  of that boundary component: namely, following \cite{ALPS}, the two  real parameters of the (non-nonegative) transverse measure $\tau$ around the boundary component  and twist parameter $\sigma$  constitute the parameter space 
 \begin{equation}\label{pr2} 
 \mathbb{R}^{[2]} = \mathbb{R}_{\geq 0} \times \mathbb{R}/\sim\text{ where } (0, \sigma) \sim (0, - \sigma)
 \end{equation}
 that is homeomorphic to $\mathbb{R}^2$. 

Here, when the transverse measure $\tau=0$, the absolute value of the $\sigma$ coordinate equals the transverse measure across the corresponding ring domain adjacent to the boundary. In particular, if the twist parameter  $\sigma$  is kept fixed, and the transverse measure $\tau \to 0$, then the foliations converge to a ring domain of length $ \lvert \sigma \rvert$. This describes the phenomenon that foliations converge to one with a ring domain, as we twist more and more, and decrease the transverse measure at the appropriate rate so that the foliations converge.  Note that one can converge to such a foliation both by positive or negative twists;  this results in the identification of the positive and negative rays as described above.

On a punctured disk $U_i$ with pole order $n_i=2$, the model foliation $F\vert_{U_i}$ is uniquely determined from the prescribed asymptotic direction and the transverse measure of $\partial U_i$  (see Proposition \ref{prop-p2}). Hence $F_0$ admits a unique extension to $U_i$, and for such a punctured disk, there are no other parameters. Note that in the case that the asymptotic direction at the $i$-th puncture forces the transverse measure of $\partial U_i$ to be zero, then the (possibly degenerate) ring domain $R$ of $F_0$ adjacent to $\partial U_i$ extends to a ring domain on $U_i$, and we ignore the transverse measure across $R$.

Adding the parameters thus obtained, we have $\chi= 6g-6 + \sum\limits_{i=1}^k (n_i +1)$ real parameters that specify $F$ uniquely.

 Conversely, any such set of $\chi$ real parameters can be realized: the parameters determine a unique measured foliation $F_0$ on $S^\prime$ by Proposition \ref{mfb}, and on each $U_i$  by Propositions \ref{prop-pn} (if $n_i>2$) and \ref{prop-p2} (if $n_i=2$). Thus, it only remains to glue the foliated disk $U_i$ by identifying the boundary $\partial U_i$ with the $i$-th boundary component of $S^\prime$; here the twist parameter $\sigma_i$ of $F_0$ associated with that boundary component plays a role.  In the case that the transverse measure of the $i$-th boundary component $\tau_i >0$, this identification of the two circles is with a twist that, in the universal cover, corresponds identifying the boundary lines after a translation by a (signed) distance $\sigma_i$, where the distance is on the line is the induced transverse measure. 
In the case the transverse measure of the $i$-th boundary component $\tau_i = 0$, the twist parameter $\sigma_i$ denotes the transverse measure across a ring domain $A_i$; here such a ring domain $A_i$ is inserted in between the boundary component of $S^\prime$ and the foliated disk $U_i$.

The topology on  $\mathcal{MF}_g(\mathfrak{n})$ is defined to make this bijection a homeomorphism: namely, a pair of foliations (with the same set of asymptotic directions at the poles) are close if 
\begin{itemize}
\item[(i)] their restriction on neighborhoods of the poles define a  pair of model foliations  that are  close, in the corresponding space $\mathcal{P}_n$, 
\item[(ii)] their restriction to the complement of these neighborhoods determines is  pair of foliations that is close in $\mathcal{MF}_{g,k}$, and 
\item[(iii)] the twist parameters that determines the gluing of each neighborhood in (i) with the complementary subsurface in (ii) are close.
\end{itemize}  \end{proof}

\noindent \textit{Remark.} Proposition \ref{mfgn-prop} assumes that the underlying surface $S$ has negative Euler characteristic;  the spaces of foliations on the complex plane $\C$ and the punctured plane $\mathbb{C}^\ast$ (relevant for Theorem \ref{thm1}) will be described in \S3 and \S4.1 respectively.

\section{The work of Au-Wan}

In this section we recall the work of Au-Wan in \cite{AuWan2} that solved the problem of prescribing horizontal and vertical foliations of a meromorphic quadratic differential on $\cp$ with exactly one pole, necessarily of order $n\geq 4$. As mentioned in \S1, the space of such quadratic differentials is
\begin{center}
$ {Q}_0(n) = \{(z^{n-4} + a_1z^{n-6} + \cdots + a_{n-2} z + a_{n-5})dz^2 \mid  a_i \in \mathbb{C} \text{ for } i=1,\ldots, n-1\}$
 \end{center}
and is thus homeomorphic to $\C^{n-5}$. Note that we have normalized our polynomial to be, in particular, monic; this fixes the asymptotic data (see Definition \ref{adata}) at the pole at $\infty$.

\subsection*{Measured foliations on $\C$} On the other hand, a measured foliation on $\cp$ with a single pole-singularity of order $n> 4$ at $\infty$ has $(n-2)$ foliated half-planes around $\infty$ and (possibly) foliated infinite strips.  In what follows we shall assume that the positive real direction is an asymptotic direction at the pole at infinity.  The leaf-space of such a foliation is thus a planar metric tree; the $(n-2)$ infinite-length edges corresponding to the half-planes are labelled by $\{1,2,\ldots, n-2\}$ in anti-clockwise order, where the ray corresponding to the positive real direction is labelled $1$. Note that this metric tree can be \textit{embedded} in $\C$, transverse to the foliation, such that each infinite ray eventually  lies in the foliated half-plane it represents.

Following  the work of Mulase-Penkava in \cite{MulPenk}, any such metric tree is obtained by a \textit{metric expansion} of a $(n-2)$-pronged star $G_{n-2}$ (where the prongs are infinite-length rays that are labelled) that replaces the central vertex of $G_{n-2}$ by a tree (with each new vertex of degree greater than two) that connects with the rest of the graph.
They proved:

\begin{thm}[Theorem 3.3 of \cite{MulPenk}]\label{mex} The space of metric trees  $\mathsf{T}(n-2)$ with $(n-2)$ infinite rays, labelled in cyclic order, and all vertices of valence at least $3$,  is homeomorphic to $\mathbb{R}^{n-5}$.
\end{thm}

\noindent \textit{Remark.} It is easy to check that a generic tree in $\mathsf{T}(n-2)$ is trivalent at each vertex, and has exactly $n-5$ edges of finite length. These  (non-negative) lengths form parameters that parametrize a subset of  $\mathsf{T}(n-2)$ corresponding to a fixed combinatorial type; there are Catalan number of types that are obtained by Whitehead moves and the corresponding regions fit together to form $\mathbb{R}^{n-5}$ (see Figure \ref{combtypegraph}).\\

\begin{figure}
  \centering
\includegraphics[scale=0.32]{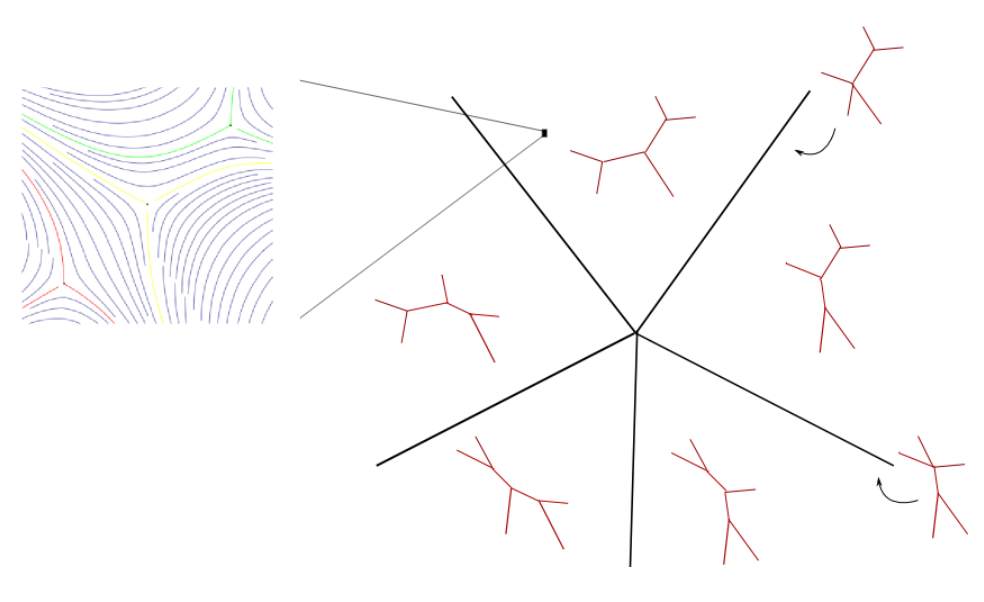}\\
  \caption{The different combinatorial types of metric trees in $\mathsf{T}(5)$ form $\mathbb{R}^2$ and parametrize the space of foliations in $\mathcal{MF}_0(5)$  (see Theorem \ref{mex} and Proposition \ref{mex2}). }
  \label{combtypegraph}
\end{figure}

As a consequence, we have:

\begin{prop}\label{mex2} For $n> 4$ the space of foliations $\mathcal{MF}_0(n)$ on $\cp$ with exactly one pole-singularity of order $n$,  is homeomorphic to $\mathbb{R}^{n-5}$.
\end{prop}
\begin{proof}
Let $\Psi_0:  \mathcal{MF}_0(n) \to \mathsf{T}(n-2)$ be the map that assigns to a foliation its leaf-space.
It is not difficult to construct an inverse map:  given a planar metric tree in $\mathsf{T}(n-2)$, we arrange foliated half-planes and foliated infinite strips  in the pattern prescribed by the tree, and identify their  boundaries. Note that the strip widths are prescribed by the edge-lengths of the tree.
The proposition then follows from Theorem \ref{mex}.
\end{proof}

\subsection*{Prescribing horizontal and vertical trees}

To complete the proof of Au-Wan's theorem stated in \S1, it remains to show that the map
\begin{equation*}
\Phi_1: Q_0(n) \to \mathcal{MF}_0(n) \times \mathcal{MF}_0(n)
\end{equation*}
 is a homeomorphism.

It suffices to define the inverse map, that is, given a pair of measured foliations (or equivalently, their metric trees), construct a holomorphic quadratic differential which has these as its vertical and horizontal foliations. Such a quadratic differential can be constructed by attaching Euclidean half-planes and bi-infinite strips to each other by isometries  (half-translations) on their boundaries; the standard differential $dz^2$ on each piece then descends to a well-defined holomorphic quadratic differential on the resulting surface (\textit{c.f.} Remark (2) after Definition \ref{fols}). 

This then becomes a combinatorial problem, which was solved by Au-Wan who gave a more general construction, that works for metric trees with countably many edges:

\begin{theorem}[Theorem 4.1 of \cite{AuWan2}]\label{aw}  Given two properly embedded planar metric trees $H,V$ in $\mathbb{C}$ and a bijection $f$ between the infinite rays of $H$ and the complementary regions of $V$, there is a unique quadratic differential on $\mathbb{C}$ or $\mathbb{D}$ with induced horizontal and vertical foliations that have leaf-spaces $V$ and $H$ respectively. Moreover, the arrangement of their foliated half-planes induces the prescribed bijection $f$.
\end{theorem}

\noindent \textit{Remarks.} (i) In the case that $V$ and $H$ have finitely-many edges (as is the case for metric trees in $\mathcal{T}(n+2)$) they showed that the resulting quadratic differential is in fact defined on the complex plane $\C$ (see Theorem 4.5 of \cite{AuWan2}).

(ii) The uniqueness of the quadratic differential obtained is clarified in Theorem 4.2 of \cite{AuWan2}: they show that if there are homeomorphisms $F,G:\mathbb{C} \to \mathbb{C}$ that restrict to isometries of $V$ and $H$ respectively,  then the quadratic differential that realizes $(V,H, f)$ is identical to the one that realizes $(V,H, G\circ f\circ F^{-1})$. \\

Thus, by Remark (ii) above,  to define the inverse of $\Phi_1$, it suffices to prescribe the bijection $f$ uniquely. We can do this by defining, assigning to each $i$ in the cyclically ordered set $\{1,2,\ldots, n-2\}$, the complementary region of $V$ that is enclosed by the infinite rays of $V$ labelled $i-1$ and $i$  (and possibly other edges of $V$).

\section{Proof of Theorem \ref{thm1}} 

In \S4.1-4.3, we shall deal with the case when $S$ is the surface $\C^\ast$, and complete the proof of Theorem \ref{thm1}.  In these sections  $n,m\geq 2$ will  be the orders of the poles at $0$ and $\infty$ respectively, such that at least one of $n,m$ is strictly greater than two, and we shall fix an asymptotic direction at each pole.  The special case when $n=m=2$ is dealt with in \S4.4.

\subsection{Foliations on $\C^\ast$}  

Following the notation introduced earlier, $\mathcal{MF}_0(n,m)$ is the space of measured foliations on $\C^\ast$ with pole-singularities of orders $n$ and $m$ at $0$ and $\infty$ that have the prescribed asymptotic data at the two poles. 

Topologically, a punctured plane can be thought of as a bi-infinite Euclidean cylinder. Any measured foliation in $\mathcal{MF}_0(n,m)$ can be decomposed into a foliation without any prong-singularities on a finite-modulus annulus $A$ in the middle of the cylinder, and  two model foliations in punctured-disk neighborhoods of the two ends, i.e.\ around $0$ and $\infty$, that lie in $\mathcal{P}_n$ and $\mathcal{P}_m$ respectively.

In what follows, the  \textit{transverse measure} $\tau_F$  of a foliation $F \in \mathcal{MF}_0(n,m)$ shall refer to the transverse measure of {the homotopy class of} a loop around the puncture(s), unless otherwise specified. (See Definition \ref{tm}.)

The leaf-space of the restriction of $F$ to $A$  is either
\begin{itemize}
\item[(a)] if $\tau_F>0$, an embedded circle homotopic to the core curve of the bi-infinite cylinder with length equal to $\tau_F$, or 
\item [(b)] if $\tau_F=0$, an embedded interval corresponding to the ring domain $A$ of length equal to the transverse measure across $A$.
\end{itemize}

The leaf-space of the entire foliation $F$ then comprises metric trees that are the leaf-spaces of the model foliations on $D_0$ and $D_\infty$ respectively, attached to the circle or interval corresponding to $A$ as above. See Figure 5. Although we shall not need this fact, we mention here that this metric graph recovers the measured foliation $F$, except when the transverse measure $\tau_F>0$,  in which case one needs the additional data of the number of Dehn-twists across $A$. \\

\begin{figure}
    \centering
 \includegraphics[scale=0.42]{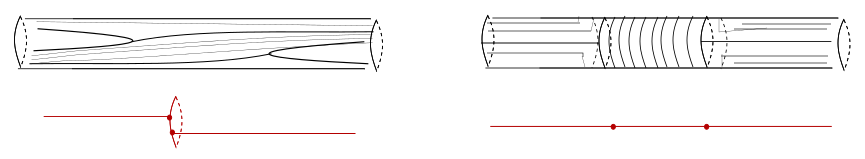}\\
  \caption{Possible measured foliations in $ \mathcal{MF}_0(3,3)$  and their leaf-spaces -- with transverse measure positive (shown on the left) and zero (shown on the right). }
  \label{singfig}
\end{figure}

For the following parametrization of $\mathcal{MF}_0(n,m)$, we shall use the above decomposition of a measured foliation on $\C^\ast$,  into model foliations on punctured disks and a foliated annulus $A$:

\begin{prop}\label{prop1} If $n,m\geq 3$, the space $\mathcal{MF}_0(n,m)$ is homeomorphic to $\mathbb{R}^{n+m-4}$. In the case that one of the poles has order $2$, say $n=2$, then  $\mathcal{MF}_0(2,m)$ is homeomorphic to $\mathbb{R}^{m-3}$ if the asymptotic direction at the order two pole is $0$ (i.e.\ the transverse measure is zero), otherwise it is homeomorphic to $\mathbb{R}^{m-1}$.
\end{prop}
\begin{proof} 
We first consider the case when $n,m>2$. Let $F\in \mathcal{MF}_0(n,m)$. 
Let $D_0$ and $D_\infty$ be neighborhoods of $0$ and $\infty$ respectively, such that $F\vert_{D_0} \in \mathcal{P}_n$ and $F\vert_{D_\infty} \in \mathcal{P}_m$. Then $A: = \mathbb{C}^\ast \setminus (D_0 \cup D_\infty)$ is a ``central annulus"  that we shall think of as a Euclidean cylinder of finite modulus.

In the case that the transverse measure $\tau_F>0$, we shall assume that all the ``twisting" of the leaves of the foliation across $\C^\ast$ happens in $A$. (For $\tau_F=0$ this twist parameter is absent, since $A$ is then a ring domain and foliations differing by a Dehn twist are isotopic to each other.)

The foliation $F\vert_A$ can be isotoped (relative the boundary) to a foliation by straight lines of constant slope; this foliation is parametrized by two real parameters, which are the transverse measure $\tau$ around $A$, and a twist parameter  $\sigma$ across $A$. These form a parameter space homeomorphic to $\mathbb{R}^2$, exactly as in equation \eqref{pr2} in the proof of Proposition \ref{mfgn-prop}. 
By Proposition \ref{prop-pn} the model foliations on $D_0$ and $D_\infty$ are parametrized by $\mathbb{R}^{n-3}$ and $\mathbb{R}^{m-3}$ respectively, assuming we have fixed the transverse measure around the boundary to be equal to $\tau_F$, since they are glued with the boundary components of $A$.  Adding the parameters, we see that the total parameter space is $\mathbb{R}^{n+m-4}$.

When one of the pole-singularities is of order two, say $n=2$, then recall that the prescribed asymptotic direction at $0$ determines the slope of the leaves on $D_0$.  We consider two sub-cases:
\begin{itemize}
\item[(a)] If the transverse measure around the pole is zero, then so is the transverse measure around $A$.  The foliations on $D_0$ and $A$ are both ring domains, and $A$ can be absorbed into $D_0$. The possible model foliations on $D_\infty$ are parametrized by $\mathbb{R}^{m-3}$ by  Proposition \ref{prop-pn}.

\item[(b)] If the transverse measure is positive, then this agrees with the transverse measure of $A$. By Proposition \ref{prop-p2}, the model foliation on $D_0$ is then completely determined, since we have already fixed the asymptotic direction at $0$.
The parameters specifying the  entire foliation then are the (positive) transverse measure around $A$, the twist parameter across $A$, and the $(m-3)$ parameters for the foliation on $D_\infty$ as before. Hence the parameter space is $\mathbb{R}^{m-1}$. 
\end{itemize} 

As in Proposition \ref{mfgn-prop}, the topology on the space $\mathcal{MF}_0(n,m)$ is defined to be the one for which this parametrization is a homeomorphism; namely, two foliations $F_1, F_2\in \mathcal{MF}_0(n,m)$ are close if their restrictions to $D_0$ and $D_\infty$ are close in the space of model foliations $\mathcal{P}_n$ and $\mathcal{P}_m$ respectively, and so are the pairs of transverse measures $\tau_1,\tau_2$ and twist parameters $\sigma_1,\sigma_2$. 
\end{proof}

\subsection{Prescribing horizontal and vertical foliations} 
Recall from \S1 that the space $Q_0(n,m)$ is the space of meromorphic quadratic differentials on $\C^\ast$ with a pole of order $n$ and $m$ at $0$ and $\infty$ respectively (we shall continue with our assumption that one of $n,m$ is greater than two), and with prescribed asymptotic directions at the poles, denoted by the set $\mathfrak{a}$.

\subsubsection*{Compatible pair} Throughout this section, $(\mathcal{H}, \mathcal{V})  \in  \mathcal{MF}_0(n,m)\times \mathcal{MF}_0(n,m)$ will be a pair of foliations, where the space of measured foliations in the first factor has prescribed asymptotic directions given by $\mathfrak{a}$, and the second factor has opposite asymptotic directions given by $\sqrt{-1}\cdot  \mathfrak{a}$ (see Definition \ref{opp}). 
We shall further assume that these two foliations are compatible in the sense defined in \S2.2  -- first, they do not both have zero transverse measure (for the non-trivial loop in $\C^\ast$ around the punctures) and second, in the case that one of the poles has order two and their transverse measures $\tau_H$ and $\tau_V$  are positive, then they are compatible for the asymptotic direction at the order two pole (see Definition \ref{compat}). 

\subsubsection*{Outline} Our goal in this section is to construct a meromorphic quadratic differential in $Q_0(n,m)$ whose horizontal and vertical foliations are $\mathcal{H}$ and $\mathcal{V}$, respectively.  To do this, we consider the decomposition of each foliation into  model foliations in the punctured-disk neighborhoods $D_0$ and $D_\infty$ of $0$ and $\infty$ respectively, and a foliation on a central annulus $A$, as in \S4.1.  

Our strategy is to first  construct 
\begin{itemize}
\item[(a)] a flat annulus $A$ that realizes the prescribed  pair of foliations $\mathcal{H}\vert_A$ and $\mathcal{V}\vert_A$, 
\item[(b)] singular flat metrics on $D_0$ and $D_\infty$,  induced by meromorphic quadratic differentials with poles of orders $n$ and $m$ respectively at the punctures, that realize the prescribed pairs of model foliations.
\end{itemize} 
Finally, we shall glue these singular-flat pieces to get the desired meromorphic quadratic differential on $\C^\ast$. 

\subsubsection*{Constructing singular flat surfaces} 
We start with describing the construction in (a) and (b) of the outline above; (a) is handled by Lemmas \ref{lem1a} and \ref{lem1b}, and (b) is handled by Lemmas \ref{lem2a} and \ref{lem2b}. 

\begin{figure}
    \centering
 \includegraphics[scale=0.45]{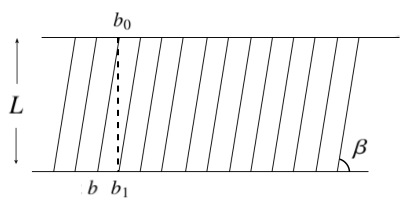}\\
  \caption{In the proof of Lemma \ref{lem1a} the flat metric on $A$ is obtained as a quotient of a strip $S(L)$ of height $L$. }
  \label{singfig}
\end{figure}

\begin{lem}\label{lem1a}
Let $t\in \mathbb{R}$ and $\mathcal{H}, \mathcal{V}$ be two measured foliations on an annulus $A$ with  positive transverse measures (around $A$)  $\tau_H$ and  $\tau_V$. Then, there is a unique flat metric (with geodesic boundary) on $A$ induced by a (constant) holomorphic quadratic differential $q$, such that the horizontal and vertical foliations of $q$ are equivalent to $\mathcal{H}, \mathcal{V}$ respectively and the difference of the twist parameters is $t$. 
\end{lem}
\begin{proof}
Passing to the universal cover, it suffices to show that there is a unique choice of $c\in \mathbb{C}^\ast$ and $L>0$ such that the desired flat annulus $A$ is the quotient of the infinite strip $S(L) = \{ w \in \C\ \vert \ 0\leq \Im (w) \leq L \}$ equipped with the quadratic differential metric $\tilde{q} = c^2dz^2$, with the infinite cyclic group $\mathbb{Z} = \langle w \mapsto w+1\rangle$. 

Let $\text{Arg}(c) = -\beta$; the transverse measures $\tau_V = \lvert c \rvert \lvert \sin \beta \rvert $ and $\tau_H = \lvert c \rvert \lvert  \cos \beta \rvert $ since they are the absolute values of the  imaginary and real parts of $\int_{[0,1]} \sqrt{\tilde{q}}$ (\textit{c.f.} the remark following Definition \ref{compat}). Thus the two transverse measures determine $\lvert c\rvert$, and an angle $\beta \in (0,\pi)$ up to an ambiguity of sign, i.e\ either $\beta$ or $\pi - \beta$.
As we shall now see, the sign of the relative twist parameter $t$ fixes the ambiguity in $\beta$, and determines the remaining parameter $L$ uniquely. 
 
Note that the horizontal foliation of $ \tilde{q}$ comprises straight lines at an angle  $\beta$ or $\pi-\beta$, and the vertical foliation comprises straight lines at an angle $\pi/2+ \beta$  or $\pi/2 - \beta$.  We shall assume that the twist parameter is measured relative to the two basepoints $b_0$ and $b_1$ on the top and bottom boundary components of $S(L)$ respectively,  that lie on the same  vertical line. Namely, the  twist parameter $\sigma_H$ of the horizontal foliation is measured as follows: consider a lift $l \subset S(L)$ of a horizontal leaf  that passes through $b_0$, and intersects the other boundary component at a point $b$; then $\sigma_H$ equals   the  (signed) transverse measure of $\mathcal{H}$ of the interval between $b_1$ and $b$ on the other boundary component.  (See Figure 6, which shows the case when the horizontal leaves make an angle $\beta$.)  We can calculate $\sigma_H$ by integrating  $\sqrt{\tilde{q}}$ along that interval, which has length $L \lvert \cot \beta\rvert$, and taking the  real part of the integral. This yields  $\sigma_H = \pm L \lvert c\rvert  \cot\beta  \cos\beta $, where the sign depends on whether the angle of the horizontal foliation is $\beta$ or $\pi-\beta$ . Similarly, the twist parameter $\sigma_V$ of the vertical foliation equals $\pm L \lvert c \rvert   \cot\beta  \sin\beta$, with the opposite dependence on the two possibilities for $\beta$.   Since the difference  $\sigma_H - \sigma_V = t$ is prescribed,   the sign of $\beta$, and $L$ are uniquely determined.
\end{proof}

\begin{lem}\label{lem1b}
Let $\mathcal{H}, \mathcal{V}$ are two measured foliations on an annulus $A$ such that 
\begin{itemize}
\item the transverse measure $\tau_H$ of $\mathcal{H}$ around $A$ is zero, i.e.\ $\mathcal{H}$ on $A$ is a ring domain, and
\item  $\mathcal{V}$ has positive  transverse measure $\tau_V>0$.
\end{itemize}
Then, there is a flat metric (with geodesic boundary)  on $A$ induced by a (constant) holomorphic quadratic differential $q$, such that the horizontal and vertical foliations of $q$ are equivalent to $\mathcal{H}, \mathcal{V}$ respectively. Moreover, $q$ is unique if we also specify the transverse measure of $\mathcal{H}$ across the annulus $A$. 
\end{lem}
\begin{proof}
As in the proof of the previous lemma, we pass to the universal cover, and consider the infinite strip  $S(L) = \{ w \in \C\ \vert \ 0\leq \Im (w) \leq L \}$ equipped with the metric induced by the quadratic differential $\tilde{q} = a^2dz^2$ for some $a\in \C^\ast$, such that $A$ is the quotient of $S(L)$ by the infinite cyclic group $\mathbb{Z} = \langle w \mapsto w+1\rangle$. 

In this case, since $\mathcal{H}$ is a ring domain, the horizontal foliation of $\tilde{q}$ is by horizontal lines in $S(L)$, and the vertical foliation is by vertical lines; consequently, $a = \tau_V \in \mathbb{R}^+$.  The remaining parameter $L$ (the ``height" of the flat annulus $A$) equals the transverse measure of $\mathcal{H}$ across $A$, if the latter is specified. 
\end{proof}

\noindent \textit{Remark.} In the case the transverse measure of $\mathcal{H}$ across $A$ is zero, the flat annulus $A$ is degenerate, i.e. is a circle of length $\tau_V$; in this case we shall still refer to $A$ as a flat metric on an annulus. \\

The construction for (b) in the outline is easier in the case the model foliations on $\mathbb{D}^\ast$ have an order two pole:

\begin{lem}\label{lem2a}  Let  $\mathcal{H}, \mathcal{V} \in \mathcal{P}_2$ be two model foliations on a punctured disk that are compatible, that is, 
\begin{itemize}
\item one of their transverse measures around the boundary of the disk is non-zero, and 
\item if both transverse measures are positive, then they are compatible in the sense of Definition \ref{compat}.
\end{itemize}
Then there exists a meromorphic quadratic differential  $q$ on $\mathbb{D}^\ast$ whose induced horizontal and vertical foliations are equivalent to $\mathcal{H}$ and $\mathcal{V}$ respectively, and the boundary circle $\partial \mathbb{D}$ is geodesic in the induced singular flat metric.  Moreover, the quadratic differential $q$ is unique if we prescribe an asymptotic direction $\theta$ at the puncture,  and in the case  both transverse  measures are positive, we require that the horizontal leaves are incident on the boundary at a prescribed angle $\beta\in (0,\pi)$.
\end{lem} 

\begin{proof}
We start with the case when one of the transverse measures equals zero, say $\tau_H = 0$ and $\tau_V>0$, where $\tau_H, \tau_V$ are the transverse measures around $\partial \mathbb{D}$ of the model foliations $\mathcal{H}$ and $\mathcal{V}$ respectively.  Passing to the universal cover $\HH$,  $\mathcal{H}$ and $\mathcal{V}$ lift to a foliation of $\HH$ by horizontal lines and vertical lines, respectively.  The quadratic differential $\tilde{q} = \tau_V^2 dw^2$ on $\HH$ has these as its horizontal and vertical foliations; this is invariant under the translations $\langle w\mapsto w+1 \rangle$ and descends to the desired  meromorphic quadratic differential  $q$ on $\D^\ast$. 

Note that in the case that $\tau_V=0$ and $\tau_H>0$, then we take $\tilde{q} =  - \tau_H^2dw^2$ on $\HH$; its horizontal and vertical foliations comprise the vertical and horizontal lines on $\HH$ respectively, and once again, this quadratic differential descends to $\D^\ast$ to define the desired $q$. 

In the case that both transverse measures $\tau_H,\tau_V>0$, then by their compatibility, we can write $\tau_H = \tau \lvert \cos \theta \rvert $ and $\tau_V = \tau \sin \theta$  for some $\tau>0$ and $\theta \in [0, \pi)$, where $\theta$ is the prescribed asymptotic direction of $\mathcal{H}$ at the order two pole. This time, in the universal cover $\mathbb{H}$ we consider the quadratic differential $\tilde{q} = a^2dw^2$ where $a = \tau e^{-i\theta}$. The horizontal and vertical foliations of $\tilde{q}$ are then  foliations by straight lines of slopes $\theta$ and $\theta + \pi/2$ (considered modulo $\pi$) respectively. In the quotient $\D^\ast = \HH/\langle w \mapsto w+1 \rangle$, the horizontal and vertical transverse measures of the boundary $\partial \mathbb{D}$ correspond to the horizontal and vertical transverse measures of the interval $[0,1] \subset \partial \HH$ which are the absolute values of the imaginary and real parts of $a$, respectively. Thus, we have obtained our desired quadratic differential $q$.  

Note that the induced metric on $\mathbb{D}^\ast$ determines semi-infinite Euclidean cylinder $E$ with geodesic boundary, and the horizontal foliation of $q$ intersects the boundary circle $\partial \mathbb{D}$ at the angle $\theta$ (which is necessarily $0$ if $\tau_H=0$). To obtain the desired angle $\beta \in (0,\pi)$ (in the case that $\tau_H>0$), we consider the cylindrical end $\overline{E}$ embedded in $E$, which is bounded by a geodesic circle $C$ chosen such that the horizontal foliation intersects $C$ at angle $\beta$. (See Figure 7.)  Since $\overline{E} \cong \mathbb{D}^\ast$, the restriction $q\vert_{\overline{E}}$ defines the desired quadratic differential. 

To show the uniqueness statement, observe that since these model foliations do not have any prong-singularities, the metric induced by $q$ is in fact flat, without any singularities. If $\partial \mathbb{D}$ is geodesic, then passing to the universal cover, we obtain a Euclidean half-plane bounded by a bi-infinite straight line, that we can realize as the upper half-plane $\HH$ equipped with a constant quadratic differential $q = a^2 dz^2$ for some $a \in \C^\ast$. The constant $a$ is uniquely determined by the prescribed asymptotic direction $\theta$ and the requirement that the transverse measures on the quotient $\D^\ast = \HH/\langle w \mapsto w+1 \rangle$ are equal to the prescribed $\tau_H, \tau_V$, exactly as described above. Finally, (in the case that $\tau_H>0$)  the sub-cylinder $\overline{E}$ bounded by the geodesic circle intersecting the horizontal foliation at angle $\beta$ is unique up to isometry. 
\end{proof}

\begin{figure}
    \centering
 \includegraphics[scale=0.33]{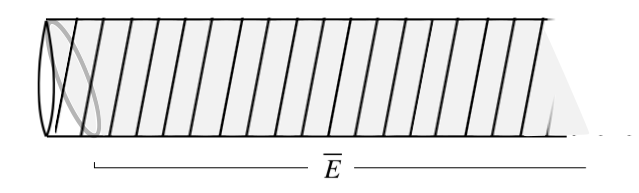}\\
  \caption{A cylindrical end corresponding to a pole of order two with positive transverse measure has an embedded sub-cylinder $\overline{E}$ (shown shaded) with geodesic boundary that intersects the horizontal foliation at angle $\beta \in (0,\pi)$. }
  \label{singfig}
\end{figure}

\medskip

For poles of higher order, the construction for step (b)  shall use the work of Au-Wan in the planar case by extending the model foliation on $\mathbb{D}^\ast$ to $\mathbb{C}^\ast$ and passing to the universal cover $\C$:

\begin{lem}\label{lem2b} Let $r > 2$ and let  $\mathcal{H}, \mathcal{V} \in \mathcal{P}_r$ be two model foliations on a punctured disk such that at least one of their transverse measures around the boundary of the disk is positive. Then there exists a unique meromorphic quadratic differential  $q$ on $\mathbb{D}^\ast$ such that 
\begin{itemize}
\item[(i)] the horizontal and vertical foliations  of $q$ are equivalent to $\mathcal{H}$ and $\mathcal{V}$ respectively, 
\item[(ii)]  the boundary circle $\partial \mathbb{D}$ is geodesic in the induced singular-flat metric, and when both transverse  measures are positive, the horizontal leaves are incident on the boundary at a prescribed angle $\beta\in (0,\pi)$,
\item[(iii)]  the induced metric has at least one prong-singularity on the boundary circle, and 
\item[(iv)] $q$ has a prescribed asymptotic direction at the pole.
\end{itemize} 
\end{lem}
\begin{proof}
The idea of the proof is to reduce to the planar case  as in  Theorem \ref{aw}, by considering the the lifts of the foliations to the universal cover $\mathbb{H}$, and then extending them to $\C$. 
Here, the universal covering map  is $\pi:\mathbb{H} \to \mathbb{D}^\ast$ defined by $\pi(w) = e^{2\pi i w}$. The group of deck-translations $\pi_1(\C^\ast) = \mathbb{Z}$ acts on $\mathbb{H}$ by the group of translations generated by $w\mapsto w + 1$. 

Let $\tau_H, \tau_V$ be the transverse measures of $\mathcal{H}$ and $\mathcal{V}$ respectively; we shall denote the lifts of the latter foliations by $\widetilde{\mathcal{H}}$ and $\widetilde{\mathcal{V}}$ respectively.

In the case the transverse measure $\tau_H>0$ (resp. $\tau_V>0$), one can isotope the leaves of $\widetilde{\mathcal{H}}$ (resp. $\widetilde{\mathcal{V}}$) so that they are orthogonal to $\mathbb{R}$, the boundary of the upper half-plane. Then we can extend $\widetilde{\mathcal{H}}$ (resp. $\widetilde{\mathcal{V}}$) to the entire complex plane $\C$ by appending the foliation by vertical lines on the lower half-plane.

On the other hand, in the case the transverse measure   $\tau_H=0$ (resp. $\tau_V=0$), the entire boundary of the upper half-plane comprises leaf segments of the lifted foliation between prong-singularities. This lifted foliation can be extended to $\C$ by appending the foliation by \textit{horizontal} lines on the lower half-plane.

\begin{figure}
    \centering
 \includegraphics[scale=0.5]{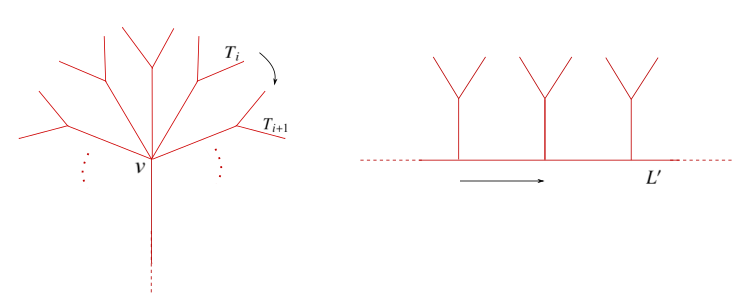}\\
  \caption{The $\mathbb{Z}$-invariant metric tree $H$ has the structure on the left if $\tau_H=0$, and the structure on the right if $\tau_H>0$. (The figure assumes $r=4$, so that the quotient by the $\mathbb{Z}$-action has two infinite rays towards the puncture at $0$.) }
  \label{singfig}
\end{figure}

Let $H, V$ be the metric trees that are the leaf-spaces for these extensions of the lifts of $\mathcal{H}, \mathcal{V}$  respectively. Both are metric trees with an action of the infinite cyclic group $\mathbb{Z}$ on them. The structure of these metric trees depends on the transverse measures; we now describe this for $H$ (see Figure 8):
\begin{itemize}
\item[(a)] If $\tau_H =0$, then the foliation on the lower half-plane corresponds to an infinite ray $L$ of $H$; the leaf-space corresponding to the chain of critical leaves constituting $\mathbb{R}$ is a single vertex $v$ that is the root of $L$. The rest of $H$ then comprises a collection of pairwise-isometric metric trees $T_i$, where $i\in \mathbb{Z}$. Each $T_i$ is rooted at $v$, and descends, via the covering map $\pi:\mathbb{H} \to \mathbb{D}^\ast$, to the metric graph that is the leaf space of $\mathcal{H}$. The generator of the group of deck-translations $\pi_1(\C^\ast) = \mathbb{Z}$ acts on $H$ by fixing $v$ (and the infinite ray $L$) and takes $T_i$ isometrically to $T_{i+1}$, for each $i\in \mathbb{Z}$. 

\item[(b)] if $\tau_H>0$, then the boundary $\mathbb{R}$ of the upper half-plane determines a bi-infinite line $L^\prime$ in $H$; note that $L^\prime$ is also the leaf-space of the foliation on the entire lower half-plane.  The rest of $H$ comprises metric trees rooted at vertices on $L^\prime$ invariant under the group $\mathbb{Z}$ of deck-translations that acts on $H$ by translations, such that the quotient is the metric graph that is the leaf-space of $\mathcal{H}$ on $\D^\ast$. Note that $L^\prime$ descends to a cycle on this metric graph of length $\tau_H$. 

\end{itemize}

The same description holds for $V$, with $\tau_V$ replacing the role of $\tau_H$ in (a) and (b) above.
As in \S3, these trees $H$ and $V$ can be topologically embedded in $\C$, via equivariant embeddings $i_H: H \to \C$ and $i_V:V\to \C$ where $\mathbb{Z}$ acts on the domain tree as described in (a) or (b) above, and on $\C$ by translations generated by $w\mapsto w+1$. 

Moreover, the infinite rays of $H$, and the complementary regions of $V$, acquire a labelling as follows:

 Let $\{\alpha_1, \alpha_2,\ldots, \alpha_{r-2}\}$ denote the cyclically ordered foliated half-planes surrounding the pole-singularity of $\mathcal{H}$ at $0$ on $\D^\ast$, where $\alpha_1$ is the half-plane whose boundary is asymptotic to the prescribed asymptotic direction $\theta$ and $\theta + 2\pi/(r-2)$.  These correspond to the complementary regions of the metric tree for $\mathcal{V}$. Lifting the labelling of the foliated half-planes to the universal cover and the extended foliation, this induces a labelling of the complementary regions of $V$ in $\C$, by the index set $\{\alpha^i_j  \mid i\in \mathbb{Z}, 1\leq j\leq r-2\}$.

 We can also label the infinite rays of the metric tree for $\mathcal{H}$ on $\mathbb{D}^\ast$ by  $\{a_1, a_2,\ldots, a_{r-2}\}$  in cyclic order, such that the label $a_1$ corresponds to the leaf-space of the half-plane $\alpha_1$ (as defined above). 
Passing to the universal cover, and its extension, we obtain a labelling of the infinite rays of $H$   by the index set  $\{a^i_j \mid i\in \mathbb{Z}, 1\leq j\leq r-2\}$.

We can now prescribe a bijection $f$ between the complementary regions of $V$ and the infinite rays of $H$ by the corresponding map of labels $\alpha^i_j \mapsto a^i_j$. Thus, we have a pair of metric trees $V$ and $H$ on $\mathbb{C}$, and a bijection $f$ between the complementary regions of $V$ and the infinite rays of $H$. By Theorem \ref{aw}, there is a singular-flat surface $\mathsf{S}$ that is conformally either $\mathbb{C}$ or $\mathbb{D}$, with horizontal and vertical foliations having metric graphs $V, H$, which induces the prescribed bijection $f$.

By construction, the bijection $f$ is $\mathbb{Z}$-equivariant. Namely, let $t_V$ be the relabelling of complementary regions of $V$ and  $t_H$ be the relabelling of infinite rays of $H$ induced by the self-homeomorphisms of $\C$ that extends the respective actions of $\mathbb{Z}$ on the trees. Then $t_H\circ f \circ t_V^{-1} = f$.  (Note that $t_V$  relabels  $\alpha^i_j$ by $\alpha^{i-1}_j$  and  $t_H$ relabels  $a^i_j$ by $a^{i-1}_j$ for each $i\in \mathbb{Z}$.) 
By the $\mathbb{Z}$-equivariance of the embeddings $i_H,i_V$ of $H$ and $V$ into $\C$, together with the uniqueness part of Au-Wan's theorem (see Remark (ii) following Theorem \ref{aw}), the singular-flat surface $\mathsf{S}$ is induced by a holomorphic quadratic differential on $\C$ that is invariant under the action of $\mathbb{Z} = \langle w \mapsto w + 1 \rangle$.  The singular-flat metric thus passes to the quotient annulus $\overline{\mathsf{S}}$, with horizontal and vertical  foliations equal to the extensions of  $\mathcal{H}$ and  $\mathcal{V}$ respectively.

Next, we show that in fact the singular-flat annulus $\overline{\mathsf{S}}$ thus obtained is conformally the punctured plane $\mathbb{C}^\ast$. Since the metric trees (and their quotients) are complete, so are the singular-flat metrics on $\overline{\mathsf{S}}$ and $\mathsf{S}$.  Moreover, it is easy to prove that the circumference of a disk of radius $R$ on $\mathsf{S}$ grows linearly with $R$:  since there is a lower bound on any arc cutting across any fundamental domain of the $\mathbb{Z}$-action on $\mathsf{S}$, such a disk will intersect at most $O(R)$ copies of the fundamental domain. In particular, such a disk will contain $O(R)$ singularities of the singular-flat  metric, and an application of the Gauss-Bonnet theorem then completes the argument. 
 We can then invoke the main result of \cite{Ahlfors-paper} (see also pg. 329 of \cite{Sario}) to conclude that the underlying Riemann surface is parabolic, that is, $\mathsf{S}$ is conformally equivalent to the complex plane $\mathbb{C}$.  The quotient $\overline{\mathsf{S}}$ must then be conformally equivalent to  $\mathbb{C}/\mathbb{Z}$, namely the punctured plane $\mathbb{C}^\ast$. 
 
 It remains to finally restrict to a suitable punctured disk $\D^\ast \subset \C^\ast$. For this, note that by construction, the singular flat metric on $\mathsf{S}$ has a  half-plane  $E$ where the horizontal and vertical foliations are transverse foliations without singularities, namely, the lower half-plane on $\C$. On $E$ the metric is induced by a constant quadratic differential invariant under a translation such that in the quotient we obtain a pole of order two at $\infty$ on $\C^\ast$ (\textit{c.f.} Definition \ref{sing2}) a neighborhood of which is isometric to a semi-infinite Euclidean cylinder.  If both transverse measures $\tau_H,\tau_V$ are positive,  consider the maximal (with respect to inclusion) isometrically embedded  sub-cylinder $\overline{E}$ such that the horizontal foliation intersects the geodesic boundary $\partial \overline{E}$ at a constant angle $\beta \in (0,\pi)$. (See Figure 7.)  Note that if one of the transverse measures $\tau_H$ or $\tau_V$ is zero, then the angle of intersection $\beta$ equals $0$ and $\pi/2$ respectively. 
 
The maximality of $\overline{E}$ implies that the induced singular-flat metric necessarily has prong-singularities on the boundary (otherwise we could take a larger disk). Excising $\overline{E}$, we are left with a singular-flat metric on a conformal puctured disk $\D^\ast$. The corresponding meromorphic quadratic differential on $\D^\ast$ is the desired $q$, whose horizontal and vertical foliations are, by our construction,  $\mathcal{H}$ and $\mathcal{V}$ respectively, and properties (i)-(iii) in the statement of the Lemma are satisfied. 

This singular-flat metric on $\mathbb{D}^\ast$ satisfying properties (i)-(iii)  is unique up to isometry: given any other such singular-flat  punctured disk $D^\prime$ we can pass to the universal cover $\mathbb{H}$ and attach a Euclidean half-plane $E^\prime$  by an isometry on the boundary line (which is a straight line by property  (ii)). Here, $E^\prime$ is equipped with the metric induced by a constant quadratic differential whose horizontal foliation intersects $\partial E^\prime$ at the prescribed angle $\beta \in (0, \pi)$ if both transverse measures are positive, and at an angle $0$ or $\pi/2$ otherwise. Thus, we obtain a singular-flat surface $\mathsf{S}^\prime$ realizing $H$ and $V$. By the uniqueness part of Theorem \ref{aw}, $\mathsf{S}^\prime$ is isometric to $\mathsf{S}$ via an equivariant isometry, and the quotient surface  $\overline{\mathsf{S}^\prime}$ is isometric to $\overline{\mathsf{S}}$. This isometry takes the quotient $\overline{E^\prime}$ of the half-plane $E^\prime$, to the maximal semi-infinite Euclidean cylinder $\overline{E}$ we had defined on $\overline{\mathsf{S}}$. (Such a maximal semi-infinite cylinder $\overline{E}$ in a cylindrical end is unique when we fix the prescribed angle $\beta$.)  Hence the isometry restricts to one between $\overline{\mathsf{S}^\prime} \setminus \overline{E^\prime}$ and $\overline{\mathsf{S}} \setminus \overline{E}$, that is, $D^\prime$ is isometric to the punctured disk with the singular flat metric induced by $q$. Since a conformal map between punctured disks is a rotation, the isometry equals the identity map if property (iv) is satisfied, that is, we prescribe the asymptotic direction of $q$ at the puncture. \end{proof}

\medskip 

We can now prove the main result of this section:

\begin{prop}\label{prop2} Let $\mathcal{H}$ and $\mathcal{V}$ be a compatible pair of measured foliations, as introduced in the beginning of \S4.2. Then there exists a unique meromorphic quadratic differential in $Q_0(n,m)$ that has horizontal foliation equivalent to $\mathcal{H}$ and vertical foliation equivalent to $\mathcal{V}$.

\end{prop}
\begin{proof}
 From the proof of Proposition \ref{prop1}, the foliations $\mathcal{H}$ and $\mathcal{V}$ decompose into model foliations $H_0,V_0 \in \mathcal{P}_n$ respectively in an open punctured-disk neighborhood  $D_0$ of $0$, into model foliations $H_\infty,V_\infty \in \mathcal{P}_m$ respectively in an open  punctured-disk neighborhood  $D_\infty$ of $\infty$, and a foliated annulus $A$ inbetween, such that $\C^\ast = D_0 \cup A\cup D_\infty$ for both. 
 
 Using Lemma \ref{lem1a} or \ref{lem1b}, there is a unique flat metric on $A$ determined by the parameters of the restrictions $\mathcal{H}\vert_A$ and $\mathcal{V}\vert_A$, such that they are the horizontal and vertical foliations, respectively, with the prescribed \textit{relative} twist parameter (i.e\ the difference of the twist parameters, as in Lemma \ref{lem1a}), and the boundary components are geodesic circles. From the proofs of these lemmas,  the angle at which the horizontal foliation on $A$ intersects the boundary components is uniquely determined; we call this angle $\beta$. 
  
Using Lemma \ref{lem2a} or \ref{lem2b}, there are uniquely defined singular-flat metrics on $D_0$ and $D_\infty$ respectively, such that 
\begin{itemize}
\item[(a)] they realize the pairs $(H_0,V_0)$ and $(H_\infty,V_\infty)$ respectively, as their horizontal and vertical foliations,  
\item[(b)] the boundary circle is geodesic in both cases, and the horizontal foliation intersects them at the angle $\beta$ as determined above, and 
\item[(c)] the horizontal foliations $H_0$ and $H_\infty$ have the prescribed asymptotic directions at the poles at $0$ and $\infty$ respectively. 
\end{itemize}

Finally, we glue these singular flat surfaces together to obtain the desired singular-flat metric on $\C^\ast$, and the corresponding meromorphic quadratic differential $q$ in $Q_0(n,m)$.  Since the asymptotic directions at the poles $D_0$ and $D_\infty$ are prescribed, the only freedom in this gluing is the number of Dehn-twists in the gluing of $\partial D_0$ with a boundary component of $A$, and in the gluing of $\partial D_\infty$ with the other boundary component of $A$. Since $A$ is an annulus, it is the difference of these two integers that matters to determine the final marked singular-flat structure on $\C^\ast$. (Note that this discussion is relevant  only if $A$ is not a ring domain, since otherwise all markings are equivalent by sliding around a closed leaf.)

This  integer parameter $d\in \mathbb{Z}$  can be measured in terms of the gluing in the universal cover as follows: choose an integer labelling of the fundamental domains of the $\mathbb{Z}$-action on the lifts  $\widetilde{D}_0, \widetilde{A}$ and $\widetilde{D}_\infty$, the lifts of $D_0,A$ and $D_\infty$ respectively,  and let $b_0$ and $b_1$ be a choice of basepoints on the boundary components of $\tilde{A}$ that are on the same vertical line (as in Figure 6). Then, if the gluing identifies $b_0$ with a point in the  boundary of the $r$-th fundamental domain of $\widetilde{D}_0$, and $b_1$ with a point in the  boundary of the $s$-th fundamental domain of $\widetilde{D}_\infty$, we define $d := r-s$.

Recall that  the flat metric on the central annulus $A$ takes care of the relative twist parameter of the two foliations $\mathcal{H}\vert_A$ and $\mathcal{V}\vert_A$. The actual twist parameters of $\mathcal{H}$ and $\mathcal{V}$ are then realized by choosing the integer parameter $d$ appropriately.  There is a unique such choice, and the marked singular-flat metric on $\C^\ast$, and therefore $q$, is determined uniquely. 
\end{proof}

\subsection{Proof of Theorem \ref{thm1}} 
We can now complete:

\begin{proof}[Proof of Theorem \ref{thm1}]
The image of the map $\Phi_2: Q_0(n,m) \to \mathcal{MF}_0(n,m) \times \mathcal{MF}_0(n,m)$ lies in the subspace $\mathcal{S}$ of pairs of foliations that are compatible in the sense defined in \S2.2. By Proposition \ref{prop2}, we obtain an inverse to the map $\Phi_2$ defined on $\mathcal{S}$. This implies that, in particular, $\Phi_2$ is an injective map that surjects on to $\mathcal{S}$. Note that the domain $Q_0(n,m) \cong \mathbb{R}^{2n+2m-8}$, and the the target $\mathcal{S}$ is a subspace of  $\mathcal{MF}_0(n,m) \times \mathcal{MF}_0(n,m)$, which is homeomorphic to $\mathbb{R}^N$ for some $N$. Here $N$ depends on $n,m$ and, in the case that one of them equals $2$, the asymptotic direction at that pole. (In particular, $N = 2n+2m-8$ if both $n,m>2$.) 

The continuity of the map $\Phi_2$, or more generally,  the map $\Phi$ from the space of quadratic differentials to the space of measured foliations on any surface that assigns the induced horizontal (or vertical) foliation to any quadratic differential,  is a standard fact. (In the case of a closed surface, see for example the proof of Theorem 4.7 of \cite{Kerckhoff}, and the references therein.)  Briefly, since the topology of the leaf-space, i.e.\ the combinatorial structure of the metric tree of a measured foliation,  is locally constant, 
the induced horizontal and vertical foliations are locally determined by the corresponding transverse measures. The latter, in turn, are determined by the real and imaginary parts of the periods $\int_\gamma \sqrt q$,  where $\gamma$ varies over a collection of arcs between the prong-singularities and homotopically non-trivial simple closed curve on the underlying surface.  The continuity of $\Phi$ then follows from the continuity of the relative period map, defined on the space of ``framed" quadratic differentials on the surface (see, for example, Theorem 4.12 of \cite{Bridgeland-Smith}).

Hence by the Invariance of Domain, the map $\Phi_2$ is a homeomorphism onto its image.
\end{proof} 

\noindent \textit{Remark.} A consequence of this is that the subspace  $\mathcal{S}$ of compatible pairs of foliations, is homeomorphic to $\mathbb{R}^{2n+2m-8}$; this can be verified independently, by analyzing the corresponding parameter space, as in \S4.1.

\subsection{The case when $n=m=2$} 

In the special case where both poles have order two,  the  meromorphic quadratic differential $q$  on $\C^\ast$ is necessarily of the form $q = \frac{a^2}{z^2} dz^2$ where $a\in \C^\ast$. In this case the asymptotic directions at the two poles must be the same, and equal to $-\text{Arg}(a)$ (\textit{c.f.} Definition \ref{sing2}). Thus the space $Q_0(2,2)$ of such quadratic differentials with prescribed (and necessarily equal) asymptotic directions  at the poles is homeomorphic to $\mathbb{R}^+$, which can be thought of as the remaining parameter $\lvert a \rvert$. 

Let $\mathcal{MF}_0(2,2)$ be the space of measured foliations on $\C^\ast$ with pole-singularities of order $2$ at $0$ and $\infty$, and with prescribed  (and equal) asymptotic directions at the poles.  By Proposition \ref{prop-p2}, a measured foliation in  $\mathcal{MF}_0(2,2)$ is determined by the transverse measure around $\C^\ast$.  Note that this transverse measure is zero if the asymptotic directions are $0$, and the foliation lifts to a foliation on $\C$ by horizontal lines. 
We then have:

\begin{lem} A pair $(\mathcal{H}, \mathcal{V}) \in \mathcal{MF}_0(2,2) \times \mathcal{MF}_0(2,2)$  (where the prescribed asymptotic directions in the first and second factor are opposite) is realizable as the horizontal and vertical foliations of some $q\in Q_0(2,2)$  if and only if either (a) exactly one of the transverse measures is zero, and (b) both transverse measures are positive, and compatible in the sense of Definition \ref{compat}.   
\end{lem}
\begin{proof} 
The necessity of either (a) or (b) being satisfied, follows from the compatibility of the horizontal and vertical foliations (see \S2.2.). 
In the other direction, a meromorphic quadratic differential $q \in Q_0(2,2)$  is obtained in either case as follows:

Let $\tau_H, \tau_V$ be the transverse measures around $\C^\ast$ of $\mathcal{H}, \mathcal{V}$ respectively. 

If $\tau_H=0$ and $\tau_V>0$,  the quadratic differential $\tilde{q} = \tau_V^2 dw^2$ on $\C$, is invariant under the group of translations $\mathbb{Z} = \langle w \mapsto w+1 \rangle$ and defines the desired quadratic differential $q$ on the quotient $\C^\ast = \C/\mathbb{Z}$.  If $\tau_V =0$ and $\tau_H>0$, then the quadratic differential $\tilde{q} =  - \tau_H^2 dw^2$ on $\C$ descends to the required quadratic differential $q$ on $\C^\ast = \C/\mathbb{Z}$.  This handles the case (a). 

Finally, for (b), recall from the compatibility of transverse measures that $\tau_V = \tau \sin\theta$ and  $\tau_H = \tau \lvert \cos \theta \rvert$ for some $\tau>0$ and $\theta \in (0, \pi)$ is the asymptotic direction at the poles. The quadratic differential $\tilde{q} = a^2 dz^2$ where $a = \tau e^{-i\theta} $ descends to the desired $q$ on $\C^\ast$.
\end{proof}

\section{Proofs of Theorems \ref{thm2} and  \ref{thm3}} 
In this section, let $S$ be a surface of genus $g$ and $k\geq 1$ labelled punctures, where $2-2g-k<0$, that is, $S$ has negative Euler characteristic. We fix a $k$-tuple $\mathfrak{n} = (n_1,n_2,\ldots, n_k)$ such that each $n_i\geq 2$.
Our proofs shall use some of the constructions of singular-flat metrics described in \S4.2.

\subsection{Proof of Theorem \ref{thm2}} 
From the statement of  Theorem \ref{thm2}, we are given a pair of measured foliations $(\mathcal{H}$, $\mathcal{V}) \in \mathcal{MF}_g(\mathfrak{n}) \times \mathcal{MF}_g(\mathfrak{n})$ where the set of asymptotic directions of the measured foliations in the first and second factors are $\mathfrak{a}$ and $\sqrt{-1}\cdot  \mathfrak{a}$ respectively. We also know that the pair $\mathcal{H}, \mathcal{V}$ are compatible in the sense of Definition \ref{compat2}. Our task, then, is to construct a meromorphic quadratic differential $q \in Q_g(\mathfrak{n})$ whose horizontal and vertical foliations are (equivalent to) $\mathcal{H}$ and $\mathcal{V}$ respectively. 
We shall do this in the following the same strategy as the construction in Proposition \ref{prop2} in \S4.2.

\begin{proof}[Proof of Theorem \ref{thm2}] 
From the proof of Proposition \ref{mfgn-prop}, the surface $S$ can be decomposed into punctured-disk neighborhoods $\{U_i\}_{1\leq i\leq k}$ of each puncture, a surface-with-boundary  $S^\prime = S \setminus (U_1\cup U_2 \cup \cdots U_k)$. The measured foliations $\mathcal{H}$ and $\mathcal{V}$ restrict to measured foliations on $S^\prime$ that we denote by $H_0$ and $V_0$ respectively. Moreover, on each $U_i$ for $1\leq i\leq k$, the restrictions $\mathcal{H}\vert_{U_i}$ and $\mathcal{V}\vert_{U_i}$ are model foliations $H_i,V_i \in \mathcal{P}_{n_i}$. Moreover, the restrictions $\mathcal{H}\vert_{A_i}$ and $\mathcal{V}\vert_{A_i}$ define foliations $H_i^0,V_i^0$ on an annulus $A_i$ that is a collar of the boundary circle $\partial U_i$. 

As in the proof of Proposition \ref{prop2}, by Lemmas \ref{lem1a} and \ref{lem1b}, we can construct a flat metric on each $A_i$ with horizontal and vertical foliations $H_i^0,V_i^0$ respectively. Similarly, by Lemmas \ref{lem2a} and \ref{lem2b}, we can construct a singular-flat metric on $U_i \setminus A_i \cong \mathbb{D}^\ast$  (induced by a quadratic differential $q_i$ with a pole of order $n_i$ at the puncture) whose horizontal and vertical foliations are $H_i$ and $V_i$ respectively.  For each $i$, we call this singular-flat punctured disk $D_i$. It follows from the proofs of these Lemmas that one can choose each $q_i$ such that the boundary component shared by $A_i$ and $D_i$ is geodesic of the same length, such that the prescribed horizontal foliations intersect it at the same angle. We also impose that $q_i$ has an asymptotic direction at pole-singularity at the $i$-th puncture given by the corresponding entries of $\mathfrak{a}$. By the uniqueness statements in these Lemmas, the set of singular-flat annuli and punctured-disks thus obtained, are uniquely determined by $\mathcal{H}$ and $\mathcal{V}$. 

On the surface-with-boundary $S^\prime$, we can construct a singular-flat metric realizing the pair $H_0$ and $V_0$, by reducing to the compact surface case by a doubling across the boundaries.  Namely, consider the closed surface $\hat{S}$ obtained by taking two copies of $S^\prime$, and identifying the corresponding boundary components such that the closed surface obtained is orientable. This identification along the boundary components does not involve any further twist;  if $\gamma_i$ is the simple closed curve arising from the $i$-th boundary component after identification, then there is a diffeomorphism $\phi: \hat{S}\to \hat{S}$ of order two, that fixes each $\gamma_i$ pointwise and locally, is a reflection across them. We can assume, after an isotopy, that for each boundary component $\partial U_i$ of $S^\prime$, the foliations $H_0$ and ${V}_0$ are either orthogonal to $\partial U_i$ or parallel to it. Let $\widehat{H_0}$ and $\widehat{V_0}$ be the measured foliations on $\hat{S}$, invariant under $\phi$, obtained by doubling $H_0$ and $V_0$ respectively
Note that since the original foliations $\mathcal{H}$ and $\mathcal{V}$ are compatible, the transverse measures of $H_0, V_0$ around $\partial U_i$ cannot both be zero; hence the measured  foliations $\widehat{H_0}$ and $\widehat{V_0}$ we obtain on the closed surface $\hat{S}$ are transverse.  Then there exists a unique holomorphic quadratic differential $\hat{q}$ (with respect to some complex structure) on the closed surface $\hat{S}$, whose horizontal and vertical foliations are equivalent to  $\widehat{H_0}$ and $\widehat{V_0}$ respectively.  (See, for example, the proof of Theorem 4.7 of \cite{Kerckhoff} and the references therein.) Since these prescribed foliations are invariant under $\phi$, it follows from the uniqueness that $\phi$ is an involutive isometry on the induced singular-flat surface. In particular, the quotient by $\phi$ yields a singular-flat metric on $S^\prime$ with geodesic boundary, whose horizontal and vertical foliations are equivalent to $H_0$ and $V_0$ respectively.  

Note that in our preceding construction, the boundary components of singular-flat metric on $S^\prime$ are either completely horizontal (in the case the transverse measure of $H_0$ around it is zero)  or completely vertical  (in the case that transverse measure is positive). In what follows, we show how to further ensure that the horizontal foliation intersects the $i$-th boundary component at the same angle, as that of the horizontal foliation on the flat $A_i$ that was constructed earlier. Note that this modification is needed only if the transverse measures of both $H_0$ and $V_0$ around the $i$-th boundary component are positive; we call the desired angle $\beta_i\in (0,\pi)$. 

For each $i$, take a semi-infinite Euclidean cylinder $R_i$, such that the boundary $\partial R_i$ is either completely horizontal or completely vertical (matching with the $i$-th boundary component $C_i$ on $S^\prime$), and identify $\partial R_i$ with $C_i$ with an isometry  that does not introduce any further twists.  We thus obtain a complete singular-flat surface $\hat{S}$ with cylindrical ends; not that the horizontal and vertical foliations extend to the whole surface. Now, for each $i$, consider the maximal (with respect to inclusion) open semi-infinite Euclidean cylinder $E_i$ that is  isometrically embedded in the $i$-th end, such that the horizontal foliation intersects the geodesic boundary $\partial E_i$ at an angle $\beta_i$ (\textit{c.f.} Figure 7). Excising each $E_i$, we obtain the desired singular-flat surface $S^{\prime\prime}= \hat{S} \setminus (E_1 \cup E_2 \cup \cdots E_k)$ that realizes the horizontal and vertical foliations $H_0,V_0$, but now has the horizontal foliation intersecting each boundary component at a prescribed angle. Note that it is possible that $S^{\prime\prime}$ is topologically not a surface, but has degeneracies; this happens in the case that the closures of, say $E_i$ and $E_j$ intersect along a common boundary arc  (\textit{c.f.} the remark following Lemma \ref{lem1b}). In that case, we shall continue to call $S^{\prime\prime}$ a singular-flat surface, despite such degeneracies. 

\begin{figure}
    \centering
 \includegraphics[scale=0.53]{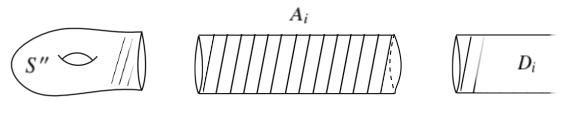}\\
  \caption{The decomposition used in the proofs of Theorems \ref{thm2} and \ref{thm3}. (The figure assumes that there is only one puncture.)}
  \label{singfig}
\end{figure}

It remains to glue these singular-flat surfaces $\{D_i, A_i\}_{1\leq i\leq k}$ and $ S^{\prime\prime}$  along their respective boundaries, as determined by the decomposition of $S$ into $\{U_i \setminus A_i, A_i\}_{1\leq i\leq k}$ and $S^\prime$,  to obtain the singular-flat metric on $S$ with horizontal and vertical foliations $\mathcal{H}$ and $\mathcal{V}$. (See Figure 9.) As in the last part of the proof of Proposition \ref{prop2}, the only freedom in this gluing is an integer parameter $d_i \in \mathbb{Z}$ that measures the relative twist between  $D_i$ and $S^{\prime\prime}$ across $A_i$, calculated as the difference in the number of Dehn-twists in the gluings of the corresponding boundary components. Recall that the flat metric on $A_i$ realizes the difference of the twist parameters of $\mathcal{H}$ and $\mathcal{V}$ at the $i$-th puncture; each $d_i$ is chosen such that the singular-flat metric on $S$ has the correct marking, and realizes the actual twist parameters for these foliations. 

The singular-flat subsurfaces we obtained were unique, and so is the integer twist parameter $d_i$ for each $i$, such a singular-flat metric  on $S$, and hence the corresponding quadratic differential $q \in Q_0(n,m)$ is unique.
\end{proof} 

\medskip

\subsection{Proof of Theorem \ref{thm3}} 

We now prove the analogue of the Hubbard-Masur Theorem (\cite{HubbMas}) for meromorphic quadratic differentials. As discussed in \S1, in contrast with the version proved in the work of Gupta-Wolf, our result dispenses with the need to choose a coordinate disk around each puncture.  The proof below uses the constructions in \S4.2 to reduce to the case when all poles are of order two, where one can use the main result of \cite{GuptaWolf0}. Note that the latter result does not depend on such a choice of coordinate disk either, since the ``residue" at a pole of order two is coordinate-independent. 

In what follows, we shall fix $X\in  \widehat{\T}_{g,k}$;  recall that $X$ represents a Riemann surface structure on the punctured surface $S$, that we shall denote by $\overline{X}$, together with the additional data of a real twist parameter at each puncture, that we record as a $k$-tuple $\mathfrak{S} = (s_1, s_2,\ldots, s_k)$.  

Recall that throughout this paper,  markings of $S$ are considered up to an isotopy that fixes a ``framing" at each puncture, or alternatively, fixes (pointwise) the boundary circles obtained by a real blow-up at each puncture.  The twist parameter $s_i$ then records the data of the framing and the marking at the punctures, as follows (\textit{c.f.} Definition 3.3. of \cite{GuptaMj1}) :
\begin{itemize}
 \item[(a)] The direction of a tangent vector $v_i$ at the $i$-th puncture, or alternatively, a point on the circle obtained as a real blowup of the $i$-th puncture, given by $\text{exp}(i 2\pi s_i)$. This determines a framing, namely, the one given by  $v_i$ and $\sqrt{-1} \cdot v_i$.
  \item[(b)] The integer $\lfloor s_i \rfloor$ that denotes the number of Dehn twists about the $i$-th puncture. 
 \end{itemize}

Recall from \S1 that $\pi$ is the projection from $Q_g(\mathfrak{n})$ to  $X\in  \widehat{\T}_{g,k}$;  in what follows $p:  \widehat{\T}_{g,k} \to \T_{g,k}$ will be the further projection that forgets the data of the twist parameters. ( In particular, note that $p(X) = \overline{X}$.) 

As in the hypotheses of Theorem \ref{thm3}, we fix a measured foliation $\mathcal{H} \in  \mathcal{MF}_g(\mathfrak{n})$, and a choice of model foliations $F_i \in \mathcal{P}_{n_i}$ for each $1\leq i\leq k$.  The set of asymptotic directions $\mathfrak{a}$ of $\mathcal{H}$ are determined by $\mathfrak{S}$: namely, at a pole of order $n_i>2$  the asymptotic direction is exactly the tangent directions as in (a) above, and if $n_i=2$, the asymptotic direction is equal to the the angle $2\pi s_i$ (modulo $\pi$).  By compatibility,  the set of asymptotic directions of $(F_1,F_2,\ldots, F_k)$ is the opposite set $\sqrt{-1}\cdot \mathfrak{a}$ (see Definition \ref{opp}). 

Moreover, the integer parameters determined by $\mathfrak{S}$ as in (b) above, are also required to match with integer twist parameters of any measured foliation $\mathcal{F}$ on a punctured surface equipped with a marking, defined as follows:

\begin{defn}[Integer twist parameter]\label{itw}  Recall that a measured foliation $\mathcal{F}$ on $S$  restricts to a model foliation on the punctured disk $D_i$ that is a neighborhood of the $i$-th pole.  Let $S^\prime = S\setminus (D_1,D_2,\ldots, D_k)$ be the surface-with-boundary obtained by deleting these neighborhoods. 
In case the transverse measure of $\mathcal{F}$ around the $i$-th puncture is positive, the integer parameter of $\mathcal{F}$ associated with that puncture on $X$ is the number of Dehn-twists required in the gluing of $D_i$ to the corresponding boundary component of $S_i$, in order to obtain the chosen marking.   In the case that the transverse measure of $\mathcal{F}$ around the $i$-th puncture is zero, the integer twist parameter is ill-defined, so we ignore such punctures. This is because in that case, the $i$-th boundary component of $S^\prime$ is the boundary of a (possibly degenerate) ring-domain, and changing $\mathcal{F}$ by a Dehn-twist around that puncture yields an equivalent measured foliation. 
\end{defn}

Thus, in our case, we shall assume the integer twist parameters of $\mathcal{H}$ on $X$ are  equal to the integer parameters determined by $\mathfrak{S}$ as in (b) above.  Note that $\mathfrak{S}$ also determines the integer twist parameters of the \textit{vertical} foliation $\mathcal{V}$ of the meromorphic quadratic differential $q\in  Q_g(\mathfrak{n})$ that we are aiming to construct.

\begin{proof}[Proof of Theorem \ref{thm3}] 

Let $\mathcal{V}$ be the vertical foliation of the desired meromorphic quadratic differential $q\in  Q_g(\mathfrak{n})$; since the horizontal foliation of $q$ would be $\mathcal{H}$, note that by Theorem \ref{thm2}, $q$ is uniquely determined by the pair $(\mathcal{H},\mathcal{V})$. The model foliations for $\mathcal{V}$ would be the prescribed foliations $F_1,F_2,\ldots, F_k$ in punctured-disk neighborhoods $D_1,D_2,\ldots, D_k$, respectively, around the punctures. It thus remains to specify the foliation $V_0$ on the surface-with-boundary $S^\prime = S\setminus (D_1,D_2,\ldots, D_k)$. 

Recall that we want $\pi(q) = X$, where $X \in  \widehat{\T}_{g,k}$ is a punctured Riemann surface equipped with a framing at the punctures, and a marking that remembers the number of Dehn-twists around each puncture, the data of which is encoded by the set $ \mathfrak{S}$.  From the discussion above, $ \mathfrak{S}$ determines the asymptotic directions of the model foliations (of either foliation) on $D_i$, and the number of Dehn-twists in the gluing of $D_i$ to the corresponding boundary component of $S^\prime$. These in turn determine the real twist parameters of $\mathcal{H}$ and $\mathcal{V}$ at the punctures, as described in the the proof of Proposition \ref{mfgn-prop} ; we denote them by $\hat{\sigma} = (\sigma_1,\sigma_2,\ldots, \sigma_k)$ and $\hat{\sigma}^\prime = (\sigma^\prime_1,\sigma^\prime_2,\ldots, \sigma^\prime_k)$ respectively.

By Lemmas \ref{lem1a} and \ref{lem1b}, for each $1\leq i\leq k$, there is a  unique flat annulus $A_i$ with its metric induced by a constant quadratic differential,  such that its horizontal and vertical foliations have transverse measures equal to those of $F_i^H$ and $F_i$ around $\partial D_i$, and  the difference of their twist parameters equals $\sigma_i - \sigma_i^\prime$, where $\sigma_i$ and $\sigma_i^\prime$ are as defined above.  Let $\beta_i \in [0,\pi)$ be the angle at which the horizontal foliation intersects the boundary components of $A_i$.  Note that if one of the transverse measures around the $i$-th pole is zero, then $\beta_i$ is necessarily $0$ or $\pi/2$. 

By Lemma \ref{lem2a} (if $n_i=2$) or \ref{lem2b} (if $n_i>2$), there is a unique singular-flat metric on each $D_i$, induced by a meromomorphic quadratic differential $q_i$, such that  
\begin{itemize}
\item[(a)] the horizontal and vertical foliations of $q_i$ are $F_i^H$ and $F_i$ respectively,  
\item[(b)] the asymptotic directions of $q_i$ are those prescribed by $\mathfrak{a}$, and 
\item[(c)] the horizontal foliation $F_i^H$ intersects the geodesic boundary $\partial D_i$ at an angle $\beta_i$, as defined above. 
\end{itemize} 

From the proofs of these Lemmas, the length of the boundaries of $A_i$ and $D_i$ are equal (they only depend on the transverse measures of $F_i^H$ and $F_i$ around $\partial D_i$, and the angle $\beta_i$). We identify each $\partial D_i$ with one of the boundary components of $A_i$ by an isometry to obtain a singular-flat disk that we denote by $U_i$. The only freedom is the number of Dehn-twists in this gluing; however this is determined by $\mathfrak{S}$, i.e.\ these are chosen such that the twist parameters of the foliations are precisely $\hat{\sigma}$ and $\hat{\sigma}^\prime$. Since the horizontal foliations intersect each boundary at the same angle by (c) above, the singular flat metric on $U_i$ is induced by a meromorphic quadratic differential $q_i$ on a punctured disk $\D^\ast$.

Let $\hat{\tau} = (\tau_1,\tau_2,\ldots, \tau_k)$ be the transverse measures of $F_1,F_2,\ldots, F_k$ respectively, around the boundaries of the corresponding punctured disks. Let $\mathcal{MF}_{g,k}(\hat{\tau})$ be the space of measured foliations on a compact surface of genus $g$ and $k$ labelled boundary components, such that the transverse measures around the boundary components are given by the $k$-tuple $\hat{\tau}$.  By Proposition \ref{mfb}, since we are fixing the parameters corresponding to the transverse measures of the boundary, the space $\mathcal{MF}_{g,k}(\hat{\tau}) \cong \mathbb{R}^{6g-6 + 2k}$. 

Let $H_0$ be the restriction of $\mathcal{H}$ to the surface-with-boundary $S^\prime$. Given $V_0 \in \mathcal{MF}_{g,k}(\hat{\tau})$, there is a unique  singular-flat metric on $S^\prime$  realizing the pair $(H_0,V_0)$ as its horizontal and vertical foliations, obtained by a doubling across the boundaries to get a closed surface and applying the Hubbard-Masur theorem, exactly as in the proof of Theorem \ref{thm2}.  Recall that the resulting horizontal and vertical foliations are either orthogonal or parallel to  each boundary component, depending on whether its transverse measure around the boundary is zero or positive, respectively.  However, we can make the horizontal foliation at the $i$-th boundary component intersect at the angle $\beta_i$  exactly as in the proof of Theorem \ref{thm2}, namely by appending a cylindrical end $R_i$, and truncating it along a suitable open sub-cylinder $E_i$ bounded by a geodesic circle  that intersects the horizontal foliation at the desired angle $\beta_i$.  Let $S^{\prime\prime}$ be the resulting singular-flat surface (with possible degeneracies when the closures of the sub-cylinders intersect). Recall that in Lemmas \ref{lem2a} and \ref{lem2b}, the angle $\beta_i$ was achieved by exactly the same truncation of a cylindrical end. Hence,  the length of the resulting geodesic boundary component  $\partial E_i$ of $S^{\prime\prime}$ matches the length of the boundary of $U_i$ obtained above.

Thus, we can identify these  boundaries (i.e.\ ``cap off" the $i$-th boundary component in $S^{\prime\prime}$ by the singular-flat punctured-disk $U_i$) to  obtain a singular-flat metric on $S$ induced by a meromorphic quadratic differential $q\in  Q_g(\mathfrak{n})$. (See Figure 9.)  Note that in this gluing we do not introduce any additional twist, since the twist parameter at the $i$-th puncture has already been taken care of by the gluing between $A_i$ and $D_i$.
Thus, by an appropriate gluing, the resulting marking on $S$ is the one on $X$, and the horizontal foliation of $q$ is exactly $\mathcal{H}$. From our construction, the vertical foliation of $q$ restricts to the desired model foliations $F_1,F_2,\ldots, F_k$ at the respective punctures.

Let $\Psi: \mathcal{MF}_{g,k}(\hat{\tau}) \to \T_{g,k}$  be the map defined by  $\Psi(V_0) = p\circ \pi (q)$ where $q\in  Q_g(\mathfrak{n})$ is the quadratic differential obtained from the construction we just described.  It only remains to show that in this construction, there is a unique initial choice of foliation $V_0$ such that the Riemann surface underlying $q$ is the desired one, i.e.\ $\Psi(V_0) = \overline{X}$.  This is immediate from the following:\\

\noindent \textit{Claim. The map $\Psi$ is a homeomorphism.} \\
\textit{Proof of claim.}   The continuity of $\Psi$ follows from the fact that in the construction above,  the singular-flat surface $S^{\prime\prime}$ depends continuously on $V_0$ from the continuity of the Hubbard-Masur map. Since both the target and domain are homeomorphic to $\mathbb{R}^{6g-6 + 2k}$, it suffices to show, by the invariance of domain, that $\Psi$ is a bijection.  

Consider the map $\Psi_0:  \mathcal{MF}_{g,k}(\hat{\tau}) \to \T_{g,k}$ that ``caps off" the boundary components of $S^{\prime\prime}$ in a different way, by attaching cylindrical ends as we now describe.  Namely, given $V_0$, construct the singular-flat surface $S^{\prime\prime}$ realizing the pair $(H_0,V_0)$ exactly as above. Then, attach the boundary of a semi-infinite Euclidean cylinder $C_i$ to the $i$-th boundary component of $S^\prime$ where $C_i$ is chosen to have a circumference equal to the length of that boundary component, and $C_i$ is equipped with a holomorphic quadratic differential whose horizontal foliation intersects $\partial C_i$ at an angle $\beta_i \in (0,\pi)$, or at zero angle,  depending on whether  the corresponding transverse measure is positive or zero respectively, or equivalently, whether $\tau_i=0$ or $\tau_i>0$ respectively.  (This looks like gluing in the shaded cylinder in Figure 7.) This defines a singular-flat metric on $S$ induced by a meromorphic quadratic differential $q_0 \in Q_g(\mathfrak{n}_0)$ where $\mathfrak{n}_0 = \underbrace{(2,2,\ldots,2)}_{k\ times}$ . We then define $\Psi_0(V_0) = p \circ \pi_0 (q_0)$, where $\pi_0$ and $p$ are the forgetful projections $\pi_0: Q_g(\mathfrak{n}_0) \to \widehat{\T}_{g,k}$ and $p: \widehat{\T}_{g,k} \to \T_{g,k}$.  

Note that  $C_i$ and $U_i$ are defined by two different quadratic differentials on the punctured-disk $\mathbb{D}^\ast$ such that the induced metric on the boundary circle $\partial \mathbb{D}$ is identical; thus, the two different ``capping off" constructions in $\Psi_0$ and $\Psi$ involve exactly the same identification map on the boundary circles. Thus, the resulting surfaces are conformally the same, i.e. the punctured Riemann surfaces underlying $q_0$ and $q$ are identical, and we have $\Psi_0 = \Psi$. 

Now let $\mathcal{H}^\prime \in \mathcal{MF}_g(\mathfrak{n}_0)$ be the measured foliation on $S$ obtained by extending $H_0$ on $S^\prime$ as follows: attach cylindrical ends to the boundary components of $S^{\prime\prime}$, as above, and extend $H_0$ on each cylindrical end by geodesic lines spiralling down the end (if the corresponding  transverse measure of $H_0$ is positive) or meridional circles (if the corresponding transverse measure of $H_0$ is zero).  Note that the asymptotic directions of $\mathcal{H}^\prime$ at the $i$-th puncture is $\beta_i$.  In the above construction,  the quadratic differential $q_0$ on the punctured Riemann surface $\overline{X} = \Psi_0(V_0)$ has  (a) horizontal foliation $\mathcal{H}^\prime$, and (b) a residue at the $i$-th pole that is prescribed by $\tau_i$ and $\beta_i$. By Theorem 1.2 of \cite{GuptaWolf0}, there exists a unique  meromorphic quadratic differential $q_0$ on $\overline{X}$ satisfying (a) and (b).  The existence of such a $q_0$ implies that $\Psi_0$ is surjective: one can obtain a $V_0$ such that $\Psi_0(V_0) = \overline{X}$ by truncating the cylindrical ends of the $q_0$-metric on $S$, and defining $V_0$ to be the vertical foliation of the resulting surface-with-boundary. 
The uniqueness of $q_0$ implies that  $\Psi_0$ is injective: if  $\Psi_0(V_0) = \Psi_0(V_0^\prime) = \overline{X}$, then the corresponding quadratic differentials $q_0$ and $q_0^\prime$ obtained in the construction are identical; then so are their vertical foliations, and consequently $V_0 =V_0^\prime$.  This proves the bijectivity of $\Psi_0$, and consequently of $\Psi$, and concludes the proof.
$\qed$

\end{proof}

\bibliographystyle{amsalpha}
\bibliography{Reference}

\providecommand{\bysame}{\leavevmode\hbox to3em{\hrulefill}\thinspace}
\providecommand{\MR}{\relax\ifhmode\unskip\space\fi MR }
\providecommand{\MRhref}[2]{%
  \href{http://www.ams.org/mathscinet-getitem?mr=#1}{#2}
}
\providecommand{\href}[2]{#2}
\begin{thebibliography}{ALPS16}

\bibitem[Ahl35]{Ahlfors-paper}
Lars Ahlfors, \emph{Sur le type d'une surface de {R}iemann}, C.R. Acad. Sci.
  Paris \textbf{201} (1935), 30--32.

\bibitem[ALPS16]{ALPS}
D.~Alessandrini, L.~Liu, A.~Papadopoulos, and W.~Su, \emph{The horofunction
  compactification of {T}eichm\"{u}ller spaces of surfaces with boundary},
  Topology Appl. \textbf{208} (2016), 160--191.

\bibitem[AW06]{AuWan2}
Thomas Kwok-Keung Au and Tom Yau-Heng Wan, \emph{Prescribed horizontal and
  vertical trees problem of quadratic differentials}, Commun. Contemp. Math.
  \textbf{8} (2006), no.~3, 381--399.

\bibitem[BD10]{BranDias}
Bodil Branner and Kealey Dias, \emph{Classification of complex polynomial
  vector fields in one complex variable}, J. Difference Equ. Appl. \textbf{16}
  (2010), no.~5-6, 463--517.

\bibitem[BS15]{Bridgeland-Smith}
Tom Bridgeland and Ivan Smith, \emph{Quadratic differentials as stability
  conditions}, Publ. Math. Inst. Hautes \'{E}tudes Sci. \textbf{121} (2015),
  155--278.

\bibitem[DES]{DES}
A.~Douady, F.~Estrada, and P.~Sentenac, \emph{Champs de vecteurs polynomiaux
  sur $\mathbb{C}$ \emph{(Unpublished manuscript)}}.

\bibitem[Dia13]{Dias}
Kealey Dias, \emph{Enumerating combinatorial classes of the complex polynomial
  vector fields in {$\Bbb C$}}, Ergodic Theory Dynam. Systems \textbf{33}
  (2013), no.~2, 416--440.

\bibitem[FLP12]{FLP}
Albert Fathi, Fran{\c{c}}ois Laudenbach, and Valentin Po{\'e}naru,
  \emph{Thurston's work on surfaces}, Mathematical Notes, vol.~48, Princeton
  University Press, Princeton, NJ, 2012, Translated from the 1979 French
  original by Djun M. Kim and Dan Margalit.

\bibitem[Gar87]{Gard}
Frederick~P. Gardiner, \emph{Teichm\"uller theory and quadratic differentials},
  Pure and Applied Mathematics (New York), John Wiley \& Sons, Inc., New York,
  1987, A Wiley-Interscience Publication.

\bibitem[GM]{GuptaMj1}
Subhojoy Gupta and Mahan Mj, \emph{Meromorphic projective structures, grafting
  and the monodromy map}, {\it preprint, arXiv:1904.03804}.

\bibitem[GM91]{GardMas}
Frederick~P. Gardiner and Howard Masur, \emph{Extremal length geometry of
  {T}eichm\"uller space}, Complex Variables Theory Appl. \textbf{16} (1991),
  no.~2-3, 209--237.

\bibitem[GW17]{GuptaWolf0}
Subhojoy Gupta and Michael Wolf, \emph{Meromorphic quadratic differentials with
  complex residues and spiralling foliations}, In the tradition of
  {A}hlfors-{B}ers. {VII}, Contemp. Math., vol. 696, Amer. Math. Soc.,
  Providence, RI, 2017, pp.~153--181.

\bibitem[GW19]{GuptaWolf2}
\bysame, \emph{Meromorphic quadratic differentials and measured foliations on a
  {R}iemann surface}, Mathematische Annalen \textbf{373} (2019), no.~1,
  73--118.

\bibitem[HM79]{HubbMas}
John Hubbard and Howard Masur, \emph{Quadratic differentials and foliations},
  Acta Math. \textbf{142} (1979), no.~3-4, 221--274.

\bibitem[Jen58]{Jenk}
James~A. Jenkins, \emph{Univalent functions and conformal mapping}, Ergebnisse
  der Mathematik und ihrer Grenzgebiete. Neue Folge, Heft 18. Reihe: Moderne
  Funktionentheorie, Springer-Verlag, Berlin-G\"ottingen-Heidelberg, 1958.

\bibitem[Jen72]{Jenkins2}
\bysame, \emph{A topological three pole theorem}, Indiana Univ. Math. J.
  \textbf{21} (1971/72), 1013--1018.

\bibitem[Ker92]{Kerckhoff}
Steven~P. Kerckhoff, \emph{Lines of minima in {T}eichm\"{u}ller space}, Duke
  Math. J. \textbf{65} (1992), no.~2, 187--213.

\bibitem[MP98]{MulPenk}
M.~Mulase and M.~Penkava, \emph{Ribbon graphs, quadratic differentials on
  {R}iemann surfaces, and algebraic curves defined over {$\overline{\bold
  Q}$}}, Asian J. Math. \textbf{2} (1998), no.~4, 875--919, Mikio Sato: a great
  Japanese mathematician of the twentieth century.

\bibitem[SN70]{Sario}
L.~Sario and M.~Nakai, \emph{Classification theory of {R}iemann surfaces}, Die
  Grundlehren der mathematischen Wissenschaften, Band 164, Springer-Verlag, New
  York-Berlin, 1970.

\bibitem[Str84]{Streb}
Kurt Strebel, \emph{Quadratic differentials}, Ergebnisse der Mathematik und
  ihrer Grenzgebiete (3) [Results in Mathematics and Related Areas (3)],
  vol.~5, Springer-Verlag, 1984.

\end{thebibliography}

\end{document}